\documentclass[a4paper, 12pt]{amsart} 
\usepackage{amsthm,amsmath,amssymb,amsfonts}
\usepackage{enumerate,mathtools,proof,url,tikz,graphicx}
\usetikzlibrary{cd}
\usepackage[dvipsnames]{xcolor}
\usepackage[colorlinks=true, linkcolor=blue,  citecolor=Green]{hyperref}
\usepackage[headings]{fullpage}

\usepackage{etoolbox}
\AtBeginEnvironment{definition}{%
  \pushQED{\qed}
  
}
\AtEndEnvironment{definition}{
  \popQED
}

\theoremstyle{definition}
\newtheorem{theorem}{Theorem}[section]
\newtheorem{definition}[theorem]{Definition}

\newtheorem{lemma}[theorem]{Lemma}

\newtheorem{proposition}[theorem]{Proposition}
\newtheorem{corollary}[theorem]{Corollary}

\newtheorem{example}[theorem]{Example}
\newtheorem{remark}[theorem]{Remark}

\newtheorem{observation}[theorem]{Observation}

\newcommand{\defarrow}{\stackrel{\text{def}}{\Longleftrightarrow}}

\newcommand{\N}{\mathbb{N}}
\newcommand{\HA}{\mathbf{HA}}

\newcommand{\IQC}{\mathrm{IQC}}
\newcommand{\IZF}{\mathbf{IZF}}

\newcommand{\AxScheme}[2]{{#1}\text{-}\mathbf{#2}}

\newcommand{\LEM}{\mathbf{LEM}}
\newcommand{\DNE}{\mathbf{DNE}}

\newcommand{\CatE}{\mathcal{E}}

\newcommand{\mycoloneqq}{\coloneqq} 
\newcommand{\myeqq}{=} 

\newcommand{\Arithl}{\mathcal{L}_{\mathrm{Arith}}}
\newcommand{\Sucs}{\mathrm{s}} 
\newcommand{\Sucm}{\mathrm{s}} 
\newcommand{\univfml}[1]{\varphi_{#1}}

\newcommand{\langE}{\mathcal{L}_{\mathcal{E}}} 
\newcommand{\langH}{\mathcal{L}_{H}} 
\newcommand{\langX}{\mathcal{L}_{X}}
\newcommand{\truem}{\mathrm{true}}
\newcommand{\subinp}[1]{[\![ #1 ]\!]}

\newcommand{\transpose}[1]{\widehat{#1}}

\newcommand{\SubtoposjE}{\mathcal{E}_{j}}
\newcommand{\jshfunc}{L_{j}}
\newcommand{\jsh}[1]{L_{j} {#1}}
\newcommand{\CljSubE}{\mathrm{Cl}_{j}\mathrm{Sub}_{\mathcal{E}}}
\newcommand{\SubEj}{\mathrm{Sub}_{\mathcal{E}_{j}}}

\newcommand{\extjtrans}[1]{{#1}^{j}}
\newcommand{\shjtrans}[1]{{#1}^{L_{j}}}
\newcommand{\intjtrans}[1]{{#1}^{*}}
\newcommand{\forceinP}{\Vdash_{\mathbb{P}}}
\newcommand{\KforceinP}{\Vdash_{\mathbb{P}}^{\mathrm{K}}}

\newcommand{\LopE}{\mathrm{Lop}}
\newcommand{\DLopE}{\mathrm{DLop}}
\newcommand{\isLopE}{\mathrm{is}\textrm{-}\mathrm{lop}}
\newcommand{\inLopE}{\in\mathrm{Lop}}
\newcommand{\inDLopE}{\in\mathrm{DLop}}
\newcommand{\subLopE}{\subseteq\mathrm{Lop}}
\newcommand{\subDLopE}{\subseteq\mathrm{DLop}}
\newcommand{\LopFE}{\mathrm{LopFrm}}
\newcommand{\PosetP}{\mathbb{P}}
\newcommand{\PosetQ}{\mathbb{Q}}
\newcommand{\inPosetP}{\in\mathbb{P}}
\newcommand{\geqinP}[1]{\geq_{\mathbb{P}} {#1}}

\newcommand{\LforceinP}{\Vdash_{\mathbb{P}}^{\mathrm{L}}}
\newcommand{\singletonX}[1]{\{\, #1\,\}_{X}}
\newcommand{\intunitX}[2]{\eta_{X} [#1, #2]} 
\newcommand{\intdefofLX}[2]{\Lambda_{X} [#1, #2]} 
\newcommand{\intLX}[1]{L_{*}{X} [#1]}

\newcommand{\opPosetP}{\mathbb{P}^{\mathrm{op}}}
\newcommand{\rest}[2]{{#1}\vert_{#2}}

\newcommand{\subobclsP}{\Omega_{\mathbb{P}}}
\newcommand{\GrothcovP}{J_{\mathbb{P}}}
\newcommand{\lopP}{C_{\mathbb{P}}}
\newcommand{\PreshtopP}{\mathbf{PSh}_{\mathcal{E}} (\mathbb{P}^{\mathrm{op}})}
\newcommand{\ShtopP}{\mathbf{Sh}_{\mathcal{E}} (\mathbb{P}^{\mathrm{op}}, J_{\mathbb{P}})}

\newcommand{\Lopclosure}[1]{\overline{#1}}
\newcommand{\stlop}[1]{\mathbb{P}_{#1}}

\newcommand{\myaxiom}[1]{\fbox{#1}:}
\newcommand{\myinfrule}[1]{\fbox{#1}:}
\newcommand{\myfmlcls}{\mathcal{R}}

\newcommand{\EquivP}[1]{\mathrm{E}_{#1}[\PosetP]}
\newcommand{\intlnot}[1]{\lnot_{#1}}

\newcommand{\MonoP}[1]{\mathrm{M}_{#1}[\PosetP]}
\newcommand{\NonoP}[1]{\mathrm{N}_{#1}[\PosetP]}

\newcommand{\Trp}[3]{\mathrm{Trp}_{#1}^{#2}[#3]}
\newcommand{\Cl}[3]{\mathrm{Cl}_{#1}^{#2}[#3]}

\newcommand{\DNEAx}[1]{\lnot\lnot{#1} \rightarrow {#1}}

\newcommand{\Eff}{\mathcal{E}\!f\!f}
\newcommand{\undset}[1]{|#1|}
\newcommand{\powN}{\mathcal{P}(\mathbb{N})}
\newcommand{\intinEff}[1]{[\![#1]\!]}
\newcommand{\OmegaEff}{\mathcal{P}(\mathbb{N})}
\newcommand{\leqEff}{\sqsubseteq}
\newcommand{\eqEff}{\equiv}
\newcommand{\Rel}[1]{\mathrm{Rel}_{=_{#1}}}
\newcommand{\LopEff}{\mathrm{Lop}(\mathcal{E}\!f\!f)}
\newcommand{\DLopEff}{\mathrm{DLop}(\mathcal{E}\!f\!f)}
\newcommand{\upset}[1]{\uparrow(#1)}

\newcommand{\Klapp}{\cdot}
\newcommand{\rKlapp}[1]{\cdot^{#1}}

\newcommand{\PFuncset}{\mathrm{PF}}

\newcommand{\Tfrealat}[2]{{#1} \Vdash_{T} {#2} \mathrel{\mathbf{r}}}

\newcommand{\leqreduce}{\leq_{\mathrm{LT}}}
\newcommand{\Tjump}[1]{\emptyset^{(#1)}}

\newcommand{\Prealat}[2]{[\![ {#1}\Vdash_{\mathbb{P}} {#2} ]\!]}
\newcommand{\jreal}[1]{[\![ {#1}^{j} ]\!]}

\title{A $j$-translation with Kripke forcing relation}
\date{\today}

\author{Satoshi Nakata} 
\address{Graduate School of Informatics, Nagoya University, 
    Furo-cho, Chikusa-ku, Nagoya-shi, 464-8601, Japan}
\email{nakata.satoshi.t9@f.mail.nagoya-u.ac.jp}

\begin{document}

\begin{abstract}
    In this paper, we introduce a translation that combines the $j$-translation with Kripke forcing in the internal logic of an elementary topos.
    First, we show that our translation is sound for intuitionistic first-order logic and Heyting arithmetic.
    Furthermore, its interpretation in the effective topos provides an extension of the sheaf model of realizability introduced by de Jongh and Goodman.
    As an application, we systematically investigate translations for semi-classical axioms.
    Based on this investigation, we establish a separation result on semi-classical arithmetic, which cannot be obtained using the usual $j$-realizability.
\end{abstract}

\maketitle

\setcounter{tocdepth}{1} 
\tableofcontents

\section{Introduction}\label{sec:intro}

\subsection{$j$-translation in proof theory and topos theory}
In intuitionistic proof theory, syntactic translations, including the negative translation (or the double negation translation), have provided many insights and applications for decades.
The first application dates back to G\"{o}del's relative consistency proof for Peano arithmetic (see \cite[Chapter~2~Section~8]{ConstructivismvolI} for the history).
However, applications of such syntactic translations are not limited to consistency proofs.
In fact, typical variants such as Friedman's $A$-translation \cite{Friedman78} and the $\$ $-negative translation \cite{FlaggFriedman86} (this terminology is due to \cite{Ishihara00})
are useful for proving partial conservation results.
Systematic studies on such applications can be found in \cite{FujiwaraKurahashi23, Ishihara00, Leivant85}.
Currently, it is well known that these important translations can be uniformly described as a \emph{$j$-translation} associated with a nucleus $j$.
A detailed summary of the $j$-translation in proof theory can be found in \cite{Aczel01}.
Recently, there have been syntactic studies on an extension to G\"{o}del's system $\mathrm{T}$ \cite{Xu20} and the Kuroda-style $j$-translation \cite{vdBerg19}.

Furthermore, nucleus and its associated $j$-translation naturally appear in topos theory as well.
Lawvere and Tierney showed a one-to-one correspondence between an internal nucleus on the subobject classifier of a topos $\CatE$ and a subtopos of $\CatE$, 
while exploring logical aspects of topos theory.
In this topos-theoretic context, such an internal nucleus is called a \emph{local operator} (or \emph{Lawvere-Tierney topology}).
From the perspective of the internal logic of $\CatE$, a local operator $j$ can be regarded as a transformation on formulas in the internal language, 
and the $j$-translation describes the internal logic of the associated subtopos $\SubtoposjE$.

\subsection{Local operator as generalized oracle}
In addition to this correspondence in topos theory, 
Hyland's discovery of the effective topos $\Eff$ \cite{Hyland82} linked the $j$-translation to realizability theory.
It is known that the internal logic of $\Eff$ corresponds to the Kleene realizability interpretation, and its local operators can be regarded as generalized oracles.
For example, for a partial function $f :\subseteq \N\to\N$ on natural numbers, there exists a local operator $j_{f}$ in $\Eff$.
The internal logic of the associated subtopos $\Eff_{j_{f}}$ coincides with Kleene realizability using $f$ as an oracle (realizability relative to $f$).
In other words, the $j$-translation in $\Eff$ yields a realizability relative to a generalized oracle. 
This is often called \emph{$j$-realizability} \cite{Kihara24rethinking, LeevOosten13, vOosten14}.
This relationship between local operator and oracle has been studied extensively in categorical realizability (see \cite{vOostenbook} as a standard textbook).

\subsection{A sheaf model of realizability}\label{subsec:sheafofreal}
On the other hand, it has long been known that combining forcing with realizability relative to an oracle is a useful method in intuitionistic proof theory.
De Jongh's unpublished work in 1969 was the first to explicitly consider a sheaf model of realizability.
He introduced this method to prove a certain conservation result between Heyting arithmetic and intuitionistic predicate logic, known as de Jongh's theorem. 
Later, Goodman introduced a similar realizability notion for Heyting arithmetic in all finite types to prove a conservation result on the axiom of choice, known as Goodman's theorem \cite{Goodman78}.
In this paper, we refer to this realizability notion as \emph{de Jongh-Goodman realizability}.
Specifically, they defined this notion by combining realizability relative to a partial function $f$ on natural numbers with weak forcing on the poset $(T, \subseteq)$, 
where $T$ is a set of partial functions on natural numbers and $\subseteq$ is the extension relation.
Since then, their arguments have been applied to separation problems on variants of the Fan theorem \cite{LubarskyRathjen13} and have been investigated syntactically \cite{Coquand13, vdBergvSlooten18}.
In particular, van Oosten pointed out that de Jongh-Goodman realizability can be understood as a PCA-valued sheaf in \cite{vOostenPhD, vOosten91}, 
and discussed such realizability based on Kripke forcing in \cite[Section~4.6.2]{vOostenbook}.

\subsection{Contributions}
In this paper, we aim to understand this sheaf model of realizability from the perspective of $j$-translation.
Our main purpose is to propose a new translation that combines $j$-translation in the internal logic of a topos with Kripke forcing relation, 
and to investigate its properties from both syntactical and semantical sides.
This approach differs from van Oosten's understanding based on PCA-valued sheaf \cite{vOostenPhD, vOosten91, vOostenbook}.

First, in Section~\ref{sec:lopframe}, we define a translation $j\forceinP\varphi$ parametrized by 
a local operator $j$ and an internal subposet $\PosetP$ of local operators in the internal logic of an arbitrary elementary topos (Definition~\ref{def:forceinP}).
The local operator $j$ and the internal subposet $\PosetP$ generalize the roles of a partial function $f$ and a subposet $T$ of partial functions in de Jongh-Goodman realizability, respectively.
The first main result is that this translation is sound for intuitionistic first-order logic $\IQC$ and Heyting arithmetic $\HA$ (Theorem~\ref{thm:soundness}, Theorem~\ref{thm:soundnessHA}).

Next, in Section~\ref{sec:realizability}, we identify the interpretation of our translation $j\forceinP\varphi$ in the effective topos.
We call this \emph{$\PosetP$-realizability} (Definition~\ref{def:Preal}).
We show that $\PosetP$-realizability is an extension of de Jongh-Goodman realizability under a mild assumption (Theorem~\ref{thm:specialcase}).

Finally, in Section~\ref{sec:semiclassical}, with future proof-theoretic applications in mind, 
we systematically investigate our translation for semi-classical axioms.
As an application, we prove a separation result on semi-classical axioms in Heyting arithmetic using $\PosetP$-realizability 
(Theorem~\ref{thm:separation}: $\AxScheme{\Sigma_{n+1}}{DNE} + \lnot\lnot(\AxScheme{\Pi_{n+1}\lor\Pi_{n+1}}{DNE})$ does not imply $\AxScheme{\Pi_{n+1}\lor\Pi_{n+1}}{DNE}$ over $\HA$).
Note that most of the proof does not require computability-theoretic arguments, and is completed by calculations of the internal logic. 

We consider this separation result to be significant in two distinct senses.
First, this separation cannot be shown using the usual $j$-realizability, 
and thus it reveals an essential difference between $\PosetP$-realizability and $j$-realizability, or between our translation and the usual $j$-translation (Corollary~\ref{cor:comparison}).
Second, these axioms naturally appear in previous studies on the prenex normal form theorem, 
and the verification of their separation is of proof-theoretic significance (Remark~\ref{rem:PNFT}).

\subsection{Organization}
This paper is organized as follows.
In Section~\ref{sec:preliminaries}, we review the internal logic (higher-order logic) of a topos.
In Section~\ref{sec:intjtrans}, we introduce a $j$-translation, treating $j$ as a variable in the internal logic of a topos.
We also supplement the relationship between this $j$-translation and the translation based on the associated sheaf functor.
In Section~\ref{sec:lopframe}, we define our translation and prove its soundness.
In addition, we also discuss the relationship between our translation and the associated sheaf functor, a Kuroda-style translation, and a corresponding internal sheaf topos.
In Section~\ref{sec:realizability}, we interpret our translation in the effective topos to provide the corresponding realizability semantics.
Furthermore, we consider a concrete example and show that it is equivalent to de Jongh-Goodman realizability.
Finally, In Section~\ref{sec:semiclassical}, we turn back to the general setting and investigate the translations of the double negation elimination and the law of excluded middle, systematically.
Based on this systematic investigation, we prove a separation result on semi-classical arithmetic using $\PosetP$-realizability.

\subsection*{Acknowledgement}
The author would like to thank Benno van den Berg, Makoto Fujiwara, and Dick de Jongh for reading a draft of this paper and giving many helpful comments.
The author is also grateful to Gerard Glowacki, Koshiro Ichikawa, Akihito Kajikawa, Masamori Kaku, Takayuki Kihara, Ulrich Kohlenbach, and Jaap van Oosten for valuable discussions.
This work was partially supported by JSPS KAKENHI Grant Number 25KJ0161.

\section{Preliminaries}\label{sec:preliminaries}

\subsection{The internal language of a topos}\label{subsec:intlang}
The proofs in Section~\ref{sec:intjtrans}, \ref{sec:lopframe} and \ref{sec:semiclassical} are formalizable within a certain higher-order logic. 
To make this explicit, we review the internal language $\langE$ of an elementary topos $\CatE$ in detail.
While various conventions exist in the literature, we follow the notation of the \emph{type theory based on equality} introduced in \cite[Part~II]{LambekScott}.

Each \emph{type} in $\langE$ corresponds to an object of $\CatE$.
In particular, $1$, $\Omega$, and $PX$ denote the terminal object, the subobject classifier, and the power object of an object $X$, respectively.
We assume that each type is associated with countably many free variables.
A \emph{context} is a finite list of typed variables, typically written as $\Gamma = [x_1:X_1, \dots, x_n:X_n]$.
The concatenation of contexts $\Gamma$ and $\Gamma'$ is denoted by $\Gamma\cup\Gamma'$.
Given a context $\Gamma$, we define \emph{terms in context $\Gamma$} (written as $(t:X\quad \Gamma)$ or $t[\vec{x}]$) 
inductively as follows: 
\begin{itemize}
    \item $*:1$; 
    \item $x:X$ for each variable $x$ in $\Gamma$; 
    \item $\langle t, t'\rangle: X\times Y$ if $t:X$ and $t':Y$;
    \item $f(t):Y$ for a morphism $f\colon X\to Y$ in $\CatE$ and a term $t:X$; 
    \item $\{\, x:X \mid \varphi\,\}:PX$ if $\varphi:\Omega$ is a term in $\Gamma\cup [x:X]$.
\end{itemize}
A term of type $\Omega$ is called an ($\langE$-)\emph{formula}, and one with no free variables is called a \emph{sentence}.
The primitive formulas in the type theory based on equality are the following: 
\[\top \mycoloneqq \truem(*):\Omega, \quad 
(t =_X t') \mycoloneqq (=_X(\langle t, t' \rangle)):\Omega, \]
\[(t \in_X u) \mycoloneqq (\in_X(\langle t, u \rangle)):\Omega,\]
where
$\truem \colon 1 \to \Omega$ is the subobject classifier, 
$=_X \colon X \times X \to \Omega$ is the characteristic map of the diagonal map $\Delta_X \colon X \rightarrowtail X \times X$,
and $\in_X \colon X \times PX \to \Omega$ is the evaluation map associated with the power object.
Other logical symbols, such as the bottom symbol $\bot$, the logical connectives $\land$, $\lor$, $\rightarrow$, and the quantifiers $\forall$, $\exists$ 
are defined in terms of these primitive formulas.
For example, universal quantification is defined by:
\[\forall x:X. \varphi[x] \mycoloneqq (\{\, x:X \mid \varphi[x]\,\} =_{PX} \{\, x:X \mid \top\,\}).\]
See \cite[Part~II~Section~2]{LambekScott} for more details.
A term $t:X$ in a context $[x_1:X_1, \dots, x_n:X_n]$ is interpreted as 
a morphism $t\colon X_1\times \dots \times X_n \to X$ in $\CatE$ in the standard way.
For any morphism $\varphi\colon X\to \Omega$, we denote its corresponding subobject by $\subinp{\varphi}\rightarrowtail X$.

We say that a \emph{formula $\varphi$ holds in $\CatE$}, denoted by $\CatE\models\varphi$, 
if its interpretation is the morphism $\top\colon X_1\times \dots \times X_n \to \Omega$.
This is equivalent to saying that the corresponding subobject $\subinp{\varphi}$ is top in the subobject lattice.
This interpretation is sound for the type theory based on equality and, in particular, for intuitionistic logic.


\subsection{Other notation}\label{subsec:notation}
We collect some notation used throughout this paper.
We often omit the subscript $X$ of the membership relation $\in_{X}$ when it is clear from the context.
For any object $X$, terms $t$, $s:PX$, and formula $\varphi$, 
we use the following abbreviations: 
\begin{itemize}
    \item $\forall x\in t. \varphi \mycoloneqq \forall x:X. (x\in_{} t \rightarrow \varphi)$ 
    \item $t\subseteq s \mycoloneqq \forall x\in t. x\in_{} s$
    \item $\forall u\subseteq t. \varphi \mycoloneqq \forall u:PX. (u\subseteq t \rightarrow \varphi)$
\end{itemize}
Unless stated otherwise, $\CatE$ denotes an arbitrary elementary topos. 
Given a local operator $j\colon\Omega\to\Omega$ in $\CatE$, 
we write $\SubtoposjE$ for the associated subtopos and  
$L_{j}\colon \CatE\to\SubtoposjE$ for the associated sheaf functor.

\subsection{Arithmetic}\label{subsec:Arithmetic} 
In this paper, we also consider translations for first-order intuitionistic arithmetic.
We here follow the conventions in \cite{ABHK04, ConstructivismvolI}.
Let $\Arithl$ denote the language of arithmetic consisting of a constant $0$, the successor $\Sucs$, and function symbols for all primitive recursive functions.
Heyting arithmetic $\HA$ is a first-order intuitionistic theory consisting of the successor axiom $\forall x. \lnot (\Sucs(x)=0)$, defining equations for all primitive recursive functions, 
and the induction axiom scheme for all $\Arithl$-formulas.
If an elementary topos $\CatE$ has a natural numbers object $N = (N, 0\colon 1\to N, \Sucm\colon N\to N)$, 
then any $\Arithl$-formula $\varphi[\vec{x}]$ can be canonically interpreted as an $\langE$-formula whose variables range over $N$: 
\[\varphi [\vec{x}] \quad\mapsto\quad \varphi:\Omega \quad [\vec{x}:N].\]
We call the $\langE$-formula the \emph{standard interpretation of $\varphi$ in $\CatE$}.
It is well known that all axioms of $\HA$ hold in $\CatE$ under this interpretation.

\section{Internal $j$-translation}\label{sec:intjtrans}

\subsection{The $j$-translation in the internal language}\label{subsec:intjtrans}
To begin with, we introduce the $j$-translation in the internal language of a topos.
Unlike arguments in locale theory or in proof theory, we emphasize that the internal logic of a topos can formalize quantification ranging over local operators. 
To this end, we first recall the basics of nucleus and the associated $j$-translation in locale theory (see, for example, \cite{JohnstoneStonespaces}).

Let $H = (H, \land, \lor, \rightarrow, \leq)$ be a locale.
An endomorphism $j\colon H\to H$ is called a \emph{nucleus} if it satisfies the following conditions: 
(1) $a \leq j(a)$, (2) $jj(a) \leq j(a)$, and (3) $j(a \land b) = j(a) \land j(b)$.
Note that the third condition can be replaced by (3)$'$ $(a \rightarrow b) \leq (j(a)\rightarrow j(b))$.
Each nucleus $j$ on $H$ induces the associated sublocale $H_{j}\mycoloneqq \{\, a \in H \mid j(a) = a\,\}$ of $j$-closed elements.
Now, let us define an \emph{$\langH$-formula} as a propositional formula over $H$ constructed from $\land$, $\lor$, and $\rightarrow$ 
(this terminology is used only in this section).
Then, the $j$-translation $\extjtrans{\varphi}$ for an $\langH$-formula $\varphi$ is defined as follows: 
\[\extjtrans{p}\mycoloneqq jp \text{ if } p \text{ is an atomic } \langH\text{-formula };\]
\[\extjtrans{(\varphi\land \psi)}\mycoloneqq \extjtrans{\varphi}\land\extjtrans{\psi}; \qquad 
\extjtrans{(\varphi\lor \psi)}\mycoloneqq j(\extjtrans{\varphi}\lor\extjtrans{\psi}); \qquad 
\extjtrans{(\varphi\rightarrow \psi)}\mycoloneqq \extjtrans{\varphi}\rightarrow\extjtrans{\psi}. \]
This translation provides a uniform way to transform $\langH$-formulas into $\mathcal{L}_{H_{j}}$-formulas.

We can apply a similar argument to the subobject classifier $\Omega$ of an elementary topos $\CatE$ and $\langE$-formulas.
In a topos $\CatE$, $\Omega$ is an internal locale. 
An endomorphism $j\colon \Omega\to\Omega$ in $\CatE$ is called a \emph{local operator} if $j$ is an internal nucleus on $\Omega$.
However, a crucial difference is that the internal language $\langE$ can formalize endomorphisms on $\Omega$ using the power object $P\Omega$, which is isomorphic to the exponential $\Omega^{\Omega}$.
That is, we can treat local operators as terms of type $P\Omega$ (note that in a presheaf topos, this data is internally equivalent to a Grothendieck coverage).
This treatment enables us to handle not only local operators but also quantification over them in $\CatE$.

\begin{definition}\label{def:lop}
    Let $j$ denote a variable of type $P\Omega$, and let $\varphi:\Omega$ be an $\langE$-formula.
    We use the following notations: 
    \begin{itemize}
        \item $j\varphi \mycoloneqq (\varphi\in_{} j)$; 
        \item a formula $\isLopE[j]\colon \Omega$ is defined as follows: 
        \[\begin{aligned}
            \isLopE[j] \mycoloneqq 
            &\forall p:\Omega. (p \rightarrow jp) \land \forall p:\Omega. (j(jp) \rightarrow jp) \\
            \land &\forall p, q:\Omega. ((p\rightarrow q) \rightarrow (jp\rightarrow jq)). 
        \end{aligned}\]
    \end{itemize}
    Then a closed term $\LopE:P(P\Omega)$ is defined by $\LopE \mycoloneqq \{\, j:P\Omega \mid \isLopE[j]\,\}$. 
\end{definition}

\begin{remark}\label{rem:POmegavslop}
    Let $j\colon \Omega\to\Omega$ be an endomorphism in $\CatE$, and let $\transpose{j}\colon 1\to P\Omega$ denote its transpose.
    One can check that $j$ is a local operator in $\CatE$ if and only if $\CatE\models (\transpose{j} \in_{} \LopE)$ holds.
    The closed term $\LopE$ naturally induces a subobject of $P\Omega$: 
    \[\LopE = \subinp{\isLopE[j]} \rightarrowtail P\Omega.\] 
    This can be regarded as an internal object of all local operators in $\CatE$.
    In our arguments below, the universal quantification $\forall j\inLopE.\varphi[j]$ can be read as the universal quantification $\forall j:\LopE. \varphi[j]$ over this object $\LopE$.
\end{remark}

Here, we list some basic facts on nucleus (cf. \cite[Lemma~1]{vdBerg19}).
\begin{lemma}\label{lem:loplem}
    Let $\varphi[\vec{x}]$ be an $\langE$-formula. 
    The following implications hold.
    \begin{enumerate}
        \item $\CatE\models \forall j\inLopE \forall p, q:\Omega. ((p\rightarrow jq) \leftrightarrow j(p\rightarrow jq))$.
        \item $\CatE\models \forall j\inLopE. ((j\forall x:X.\varphi) \rightarrow (\forall x:X.j\varphi))$.
        \item $\CatE\models \forall j\inLopE \forall p, q:\Omega. (j(p\lor q) \leftrightarrow j(jp\lor jq))$.
        \item $\CatE\models \forall j\inLopE. ((\exists x:X.j\varphi) \rightarrow (j\exists x:X.\varphi))$.
    \end{enumerate}
\end{lemma}
Based on these facts, we also have the following equivalences:
\[j\forall x:X. (\varphi\rightarrow j\psi) \leftrightarrow \forall x:X. (\varphi\rightarrow j\psi), \quad 
j\exists x:X. j\varphi \leftrightarrow j\exists x:X.\varphi.\]
For example, \cite{vdBerg19} studies a $j$-translation for first-order formulas.
For simplicity, we restrict our attention to one-sorted first-order formulas.
\begin{definition}[$\langX$-formula]
    Let $X$ be an object in a topos $\CatE$.
    An \emph{$\langX$-formula} is defined inductively as follows: 
    \[\varphi, \psi \Coloneqq R[\vec{x}] \mid 
    \varphi\land\psi \mid \varphi\lor\psi \mid \varphi\rightarrow\psi \mid
    \exists x:X. \varphi \mid \forall x:X. \varphi,\]
    where $R$ is a morphism from $X^{m}$ to $\Omega$ in $\CatE$, and $m$ is the length of $\vec{x}$.
    Here, $\bot$ is also considered an $\langX$-formula, and $\lnot\varphi$ stands for $\varphi\rightarrow\bot$.
\end{definition}
In other words, an $\langX$-formula is an $\langE$-formula whose atomic formulas are given by morphisms $R\colon X^{m}\to \Omega$ in $\CatE$ 
and whose variables range over $X$.
For instance, 
if $\CatE$ has a natural numbers object $N$, 
the standard interpretation (see Section~\ref{subsec:Arithmetic}) of each $\Arithl$-formula is an $\mathcal{L}_N$-formula.

Then, the $j$-translation for $\langX$-formulas is defined internally as follows.
Note that $j$ in the following definition is a variable of type $P\Omega$ intended to represent an internal local operator.
\begin{definition}\label{def:intjtrans}
    For any $\langX$-formula $\varphi[\vec{x}]$, 
    we inductively define $\intjtrans{\varphi}$ as follows: 
    \[\begin{aligned}
        \intjtrans{(R[\vec{x}])} &\mycoloneqq jR[\vec{x}]; \\
        \intjtrans{(\varphi\land\psi)} &\mycoloneqq \intjtrans{\varphi}\land \intjtrans{\psi}; \\
        \intjtrans{(\varphi\lor\psi)} &\mycoloneqq j(\intjtrans{\varphi}\lor \intjtrans{\psi}); \\
        \intjtrans{(\varphi\rightarrow\psi)} &\mycoloneqq \intjtrans{\varphi}\rightarrow \intjtrans{\psi}; \\
        \intjtrans{(\exists y:X. \varphi[\vec{x}, y])} &\mycoloneqq j(\exists y:X. \intjtrans{\varphi}[j, \vec{x}, y]); \\
        \intjtrans{(\forall y:X. \varphi[\vec{x}, y])} &\mycoloneqq \forall y:X. \intjtrans{\varphi}[j, \vec{x}, y]. 
    \end{aligned}\]
    We call $\intjtrans{\varphi}$ the \emph{(internal) $j$-translation of $\varphi$}:
    \[\varphi:\Omega \quad [\vec{x}:X] \quad\mapsto\quad \intjtrans{\varphi}:\Omega \quad [j:P\Omega, \vec{x}:X].\]
    In addition, given a local operator $j\colon \Omega\to\Omega$ in $\CatE$, 
    we write $\extjtrans{\varphi}$ for $\intjtrans{\varphi}[\transpose{j}, \vec{x}]$, which is obtained by substituting the transpose $\transpose{j}\colon 1\to P\Omega$ into $\intjtrans{\varphi}$. 
\end{definition}

\subsection{Relationship to the associated subtopos}\label{subsec:shjtrans}
It is well known that for any topos $\CatE$, each local operator $j$ induces the associated subtopos $\SubtoposjE$ of $j$-sheaves.
This is analogous to the associated sublocale in locale theory.
Recall that for a locale $H$, the $j$-translation transforms an $\langH$-formula into an $\mathcal{L}_{H_{j}}$-formula.
However, this does not hold for the $j$-translation of $\langX$-formulas.
Since an object $X$ is not necessarily a $j$-sheaf, the $j$-translation $\extjtrans{\varphi}$ of an $\langX$-formula $\varphi$ is not an $\mathcal{L}_{\CatE_{j}}$-formula in general.
Therefore, \cite{Hyland82} defined $j$-translations that modify not only logical connectives but also the types of variables.
For comparison with our $j$-translation, let us briefly look at the translation based on the associated sheaf functor $\jshfunc\colon \CatE \to \SubtoposjE$, introduced as $\subinp{\varphi}_{j}$ in \cite[Section~5]{Hyland82}.

\begin{definition}\label{def:shjtrans}
    For any $\langX$-formula $\varphi[\vec{x}]$ and any local operator $j$ in $\CatE$, 
    we inductively define $\shjtrans{\varphi}$ as follows: 
    \[\begin{aligned}
        \shjtrans{(R[\vec{x}])} &\mycoloneqq \chi_{\jsh{\subinp{R}}} [\vec{r}]; \\
        \shjtrans{(\varphi\land\psi)} &\mycoloneqq \shjtrans{\varphi}\land \shjtrans{\psi}; \\
        \shjtrans{(\varphi\lor\psi)} &\mycoloneqq j(\shjtrans{\varphi}\lor \shjtrans{\psi}); \\
        \shjtrans{(\varphi\rightarrow\psi)} &\mycoloneqq \shjtrans{\varphi}\rightarrow \shjtrans{\psi}; \\
        \shjtrans{(\exists y:X. \varphi[\vec{x}, y])} &\mycoloneqq j(\exists s:\jsh{X}. \shjtrans{\varphi}[\vec{r}, s]); \\
        \shjtrans{(\forall y:X. \varphi[\vec{x}, y])} &\mycoloneqq \forall s:\jsh{X}. \shjtrans{\varphi}[\vec{r}, s], 
    \end{aligned}\]
    where $\chi_{\jsh{\subinp{R}}}\colon \jsh{X^{m}}\to \Omega$ denotes the characteristic map of $\jsh{\subinp{R}}\rightarrowtail \jsh{X^{m}}$.
    We call $\shjtrans{\varphi}$ the \emph{$\jshfunc$-translation of $\varphi$}:
    \[\varphi:\Omega \quad [\vec{x}:X] \quad\mapsto\quad \shjtrans{\varphi}:\Omega \quad [\vec{r}:\jsh{X}]. \qedhere\]
\end{definition}

The $\jshfunc$-translation provides a uniform way to transform $\langX$-formulas in $\CatE$ into $\mathcal{L}_{\jsh{X}}$-formulas in $\SubtoposjE$.
For instance, the $\jshfunc$-translation for specific atomic formulas or $\Arithl$-formulas is coherent with the structure of the associated subtopos, as shown below.

\begin{example}\label{ex:shjatomic}
    Let us consider an atomic $\langX$-formula $(f(\vec{x}) =_{X} g(\vec{x}))$, where $f, g\colon X^{m}\rightarrow X$ are morphisms in $\CatE$.
    Since the associated sheaf functor $\jshfunc$ preserves equalizers, 
    its $\jshfunc$-translation coincides with the equality between the morphisms $\jsh{f}, \jsh{g}\colon \jsh{X^{m}}\rightarrow \jsh{X}$:
    \[\shjtrans{(f(\vec{x}) =_{X} g(\vec{x}))} \myeqq (\jsh{f}(\vec{r}) =_{\jsh{X}} \jsh{g}(\vec{r})) \quad [\vec{r}:\jsh{X}].\]
\end{example}

\begin{example}\label{ex:shjarith}
    Let us consider an $\Arithl$-sentence $\varphi$ and its standard interpretation on a natural numbers object in $\CatE$.
    Since the associated sheaf functor $\jshfunc$ preserves the natural numbers object, 
    the $\jshfunc$-translation $\shjtrans{\varphi}$ corresponds to the standard interpretation of $\varphi$ in $\SubtoposjE$.
    That is, the following equivalence follows from the definitions: 
    \[\CatE\models \shjtrans{\varphi} \iff \SubtoposjE\models \varphi.\]
\end{example}

On the other hand, it can be shown that the $\jshfunc$-translation $\shjtrans{\varphi}$ and the $j$-translation $\extjtrans{\varphi}$ are equivalent in $\CatE$.
While this is probably a folklore among topos theorists, we have not found literature where the proof is explicitly written.
Therefore, we clarify this fact here.
The verification relies on the following lemmas.

\begin{lemma}\label{lem:pullback}
    Let $j$ be a local operator in $\CatE$, and let $\eta$ be the unit of the geometric inclusion $(\jshfunc \dashv {i}) \colon \SubtoposjE \hookrightarrow \CatE$.
    For an object $X$, let $\eta_{X}^{-1}$ denote the pullback functor along $\eta_{X}\colon X\to \jsh{X}$.
    Then the following hold.
    \begin{enumerate}
        \item For any subobject $A\rightarrowtail X$, $\eta_{X}^{-1} (\jsh{A})$ coincides with the $j$-closure $jA$ of $A$.
        \item There is a lattice isomorphism between the subobjects $\SubEj(\jsh{X})$ of $\jsh{X}$ in $\SubtoposjE$ 
        and the $j$-closed subobjects $\CljSubE(X)$ of $X$ in $\CatE$: 
        \[\eta_{X}^{-1} \colon \SubEj(\jsh{X}) \overset{{\cong}}{\longleftrightarrow} \CljSubE(X) \colon \jshfunc.\]
    \end{enumerate}
\end{lemma}

The former statement is a basic fact found in various literature. 
The latter isomorphism is proved in \cite[Proposition~2.3]{Caramello20denseness}.
In particular, for an atomic $\langX$-formula $R$, Lemma~\ref{lem:pullback} (1) immediately gives the following equality in $\CljSubE(X)$:
\[\eta_{X}^{-1} (\subinp{\shjtrans{R}}) = \eta_{X}^{-1} (\jsh{\subinp{R}}) = j\subinp{R} = \subinp{\extjtrans{R}}.\]
Furthermore, by combining Lemma~\ref{lem:pullback} (2) with basic calculations for the associated sheaf functor (for example, \cite[Theorem~5.1~(f)]{Hyland82} for universal quantification), 
we can prove the equality above holds for any $\langX$-formula $\varphi$.
This implies the following equivalence in the internal logic:

\begin{proposition}\label{prop:shvsext}
    For any $\langX$-formula $\varphi[\vec{x}]$ and any local operator $j\colon\Omega\to\Omega$ in $\CatE$, the following equivalence holds: 
    \[\CatE\models \forall \vec{x}:X. (\shjtrans{\varphi}[\eta_{X}(x_1), \dots, \eta_{X}(x_{m})] \leftrightarrow \extjtrans{\varphi}[\vec{x}] ),\]
    where $\eta_{X}\colon X \to \jsh{X}$ is the component of the unit. 
\end{proposition}

In fact, as we will see later in Section~\ref{subsec:internalKripkeframe}, 
this equivalence can also be proven directly using the internal logic (see the discussion following Proposition~\ref{prop:LforceinPvsforceinP}).

\section{A translation with an internal subposet of local operators}\label{sec:lopframe}

\subsection{Lop-frame}\label{subsec:lopframe}

There is a standard partial order on local operators.
It is well known that this order admits numerous characterizations, such as the reverse inclusion between the associated subtoposes.
What is important here is that the order can also be defined in the internal language.
We first introduce a formula expressing the standard order on local operators.

\begin{definition}\label{def:loporder}
    The formula $j\leq k$ with variables $j$, $k:P\Omega$ is defined by: 
    \[j\leq k\mycoloneqq \forall p:\Omega. (jp\rightarrow kp). \qedhere\]
\end{definition}

The formula $j\leq k$ defines an internal poset $(P\Omega, \leq)$ in $\CatE$, as well as an internal poset $(\LopE, \leq)$ of internal local operators.
Thus, each subobject $\PosetP$ of $\LopE$ induces an internal subposet of local operators: 
\[(\PosetP, \leq_{\PosetP}) \rightarrowtail (\LopE, \leq).\]

\begin{definition}\label{def:lopframe}
    Let $\PosetP$ denote a variable of type $P(P\Omega)$. 
    We say that $\PosetP$ is a \emph{lop-frame}\footnote{
        Since it is only essential for this concept to be a subposet of nuclei, it might be more appropriate to call it a \emph{nucleus-frame} depending on the context.
        We welcome any suggestions for a better terminology.
    } if the formula $(\PosetP\subLopE)$ defined in Preliminaries~\ref{subsec:notation} holds in $\CatE$.
    For any formula $\varphi[k]$ with a free variable $k:P\Omega$, we use the following abbreviation: 
    \[\forall k\geqinP{j}. \varphi[k] \mycoloneqq \forall k\in\PosetP. (j\leq k \rightarrow \varphi[k]). \qedhere\]
\end{definition}

\begin{remark}\label{rem:lopframe}
    Similarly to Remark~\ref{rem:POmegavslop}, 
    the formula $(\PosetP\subLopE)$ defines a closed term $\LopFE \mycoloneqq \{\, \PosetP: P(P\Omega) \mid \PosetP\subLopE \,\}$ 
    and induces an internal object $\LopFE$ of all lop-frames in $\CatE$.
    This object $\LopFE$ is isomorphic to the power object of $\LopE$: 
    \[\LopFE = \subinp{\PosetP\subLopE} \cong P(\LopE) \rightarrowtail P(P\Omega).\]
    Therefore, the universal quantification $\forall \PosetP\subLopE. \varphi[\PosetP]$ can be read as the universal quantification $\forall \PosetP:\LopFE. \varphi[\PosetP]$ over this object $\LopFE$.
    In Section~\ref{sec:realizability}, we will provides the explicit description of the objects $\LopE$ and $\LopFE$ in the effective topos $\Eff$.
\end{remark}

Each lop-frame gives rise to a translation that combines the $j$-translation with Kripke forcing.

\begin{definition}\label{def:forceinP}
    For any $\langX$-formula $\varphi [\vec{x}]$, 
    we inductively define $j\forceinP \varphi [\vec{x}]$ as follows: 
    \[\begin{aligned}
        j\forceinP (R[\vec{x}]) &\mycoloneqq jR[\vec{x}]; \\
        j\forceinP (\varphi\land \psi) &\mycoloneqq (j\forceinP \varphi) \land (j\forceinP \psi); \\
        j\forceinP (\varphi\lor \psi) &\mycoloneqq j((j\forceinP \varphi) \lor (j\forceinP \psi)); \\
        j\forceinP (\varphi\rightarrow \psi [\vec{x}]) &\mycoloneqq \\ 
         \forall k\geqinP{j}. &((k\forceinP \varphi [\vec{x}]) \rightarrow (k\forceinP \psi [\vec{x}])); \\
        j\forceinP (\exists y:X.\varphi[\vec{x}, y]) &\mycoloneqq j(\exists y:X. j\forceinP \varphi[\vec{x}, y]); \\
        j\forceinP (\forall y:X.\varphi[\vec{x}, y]) &\mycoloneqq \forall k \geqinP{j} \forall y:X. k\forceinP \varphi[\vec{x}, y].
    \end{aligned}\]
    \[\varphi:\Omega \quad[\vec{x}:X] \quad\mapsto\quad j\forceinP \varphi:\Omega \quad [\PosetP:P(P\Omega), j:P\Omega, \vec{x}:X]. \qedhere\]
\end{definition}

First, we observe that this translation generalizes the $j$-translation.
In fact, it is straightforward to verify that if $j$ is a maximal element in a lop-frame $\PosetP$ internally, 
the translation $j\forceinP\varphi$ coincides with the $j$-translation $\intjtrans{\varphi}[j]$.
\begin{proposition}\label{prop:maximal} 
    For any $\langX$-formula $\varphi[\vec{x}]$,
    \[\begin{aligned}
        \CatE\models \forall&\PosetP\subLopE \forall j\inLopE. \\
        &(\forall k\geqinP{j}. (j=_{P\Omega} k)) \rightarrow \forall \vec{x}:X. (j\forceinP \varphi [\vec{x}] \leftrightarrow \intjtrans{\varphi}[j, \vec{x}]).
    \end{aligned}\]
\end{proposition}


We also verify that $j$-closedness, a fundamental property of $j$-translations, holds for our translation as well.
\begin{lemma}[$j$-closedness]\label{lem:jclosed}
    For any $\langX$-formula $\varphi[\vec{x}]$, 
    \[\CatE\models \forall \PosetP\subLopE \forall j\inLopE \forall x:X. (j(j\forceinP\varphi[\vec{x}]) \leftrightarrow j\forceinP\varphi[\vec{x}] ).\]
\end{lemma}
\begin{proof}
    One can show this by induction on the complexity of $\varphi$.
    In particular, for the steps of implication $\rightarrow$ and universal quantification $\forall$,
    apply Lemma~\ref{lem:loplem} (1) and (2).
\end{proof}

\subsection{Internal Kripke frame}\label{subsec:internalKripkeframe}

In Section~\ref{subsec:shjtrans}, we observed that the relationship between the $\jshfunc$-translation $\shjtrans{\varphi}$ and the $j$-translation $\extjtrans{\varphi}$.
In fact, we can define a variant $j\LforceinP\varphi$ of our translation $j\forceinP\varphi$ that corresponds to the $\jshfunc$-translation.
This variant gives us a semantic intuition that our translation describes an internal Kripke frame in $\CatE$.
Although its practical benefit may not be apparent due to the equivalence of the translations, we discuss this variant in detail to prepare for future extensions.

To define this, let us recall the internal construction of the associated sheaf functor $\jshfunc\colon \CatE \to \SubtoposjE$ found in numerous literature \cite{JhonstoneTopostheory, SGL, Veit81}. 
This is referred to as \emph{Lawvere's construction} in \cite[Excersise~3.4]{JhonstoneTopostheory}. 
See the introduction of \cite{Johnstone74} for the history of this construction.
The key idea is to define the $j$-sheafification $\jsh{X}$ of an object $X$ as a subobject of the power object $PX$.
\begin{definition}\label{def:internalsheaf}
    We define the following $\langE$-terms: 
    \[\begin{aligned}
        \singletonX{x} &\mycoloneqq \{\, y:X \mid (y =_{X} x)\,\}:PX \quad [x:X] \\
        \intunitX{j}{u} &\mycoloneqq \{\, x:X \mid j(x \in_{X} u)\,\}:PX \quad [j:P\Omega, u:PX] \\
        \intdefofLX{j}{u} &\mycoloneqq j( \exists x:X. u=_{PX} \intunitX{j}{\singletonX{x}} ):\Omega \quad [j:P\Omega, u:PX] \\
        \intLX{j} &\mycoloneqq \{\, u:PX \mid \intdefofLX{j}{u} \,\}:P(PX) \quad [j:P\Omega] 
    \end{aligned}\]
\end{definition}

Then, for a local operator $j\colon \Omega\to\Omega$ in $\CatE$, 
the $j$-sheafification $\jsh{X}$ is given by the subobject $\intLX{\transpose{j}} = \subinp{\intdefofLX{\transpose{j}}{\cdot}} \rightarrowtail PX$.
Moreover, the component $\eta_{X}\colon X \to \jsh{X}$ of the unit of the geometric inclusion $\SubtoposjE \hookrightarrow \CatE$ is the unique morphism 
such that the following diagram commute (See, for example, \cite{Veit81}):
\begin{equation}\label{eq:shfunc}
    \begin{tikzcd}
        X \arrow[d, rightarrowtail, "\singletonX{\cdot}"'] \arrow[r, "\eta_{X}"] & 
        \jsh{X} \arrow[d, rightarrowtail] \arrow[r, "{!}"] & 
        1 \arrow[d, rightarrowtail, "\truem"] \\ 
        PX \arrow[r, "{\intunitX{\transpose{j}}{\cdot}}"'] & 
        PX \arrow[r, "{\intdefofLX{\transpose{j}}{\cdot}}"'] & 
        \Omega.
    \end{tikzcd}
\end{equation}

Furthermore, from standard calculations of internal logic, the following lemmas on the term $\eta$ are verified.
\begin{lemma}\label{lem:intunit}
    The following hold:
    \begin{enumerate}
        \item $\CatE\models \forall j, k\inLopE. (j\leq k \rightarrow \forall u\in\intLX{j}. \intunitX{k}{u} \in_{PX} \intLX{k}).$
        \item $\CatE\models \forall j, k\inLopE. (j\leq k \rightarrow \forall u:PX. \intunitX{k}{\intunitX{j}{u}} =_{PX} \intunitX{k}{u}).$
    \end{enumerate}
\end{lemma}

In particular, Lemma~\ref{lem:intunit} (1) implies that for any local operators $j$, $k\colon \Omega\to\Omega$ with $j\leq k$, 
the unit $\eta_{X}^{j, k}\colon L_{j}X \to L_{k}X$ of the geometric inclusion $\CatE_{k} \hookrightarrow \CatE_{j}$ is induced by the term $\intunitX{\transpose{k}}{\cdot} \colon PX \to PX$ as follows:
\begin{equation*}
    \begin{tikzcd}
        L_{j}X \arrow[d, rightarrowtail] \arrow[r, "\eta_{X}^{j, k}"] &
        L_{k}X \arrow[d, rightarrowtail] \\
        PX \arrow[r, "{\intunitX{\transpose{k}}{\cdot}}"'] &
        PX.
    \end{tikzcd}
\end{equation*}

This allows us to describe all information on the $j$-sheafification $\jsh{X}$ and the unit $\eta_{X}^{j, k}\colon L_{j}X \to L_{k}X$ by embedding them into $PX$.
Therefore, given a lop-frame $\PosetP$, we can define a formula that intuitively represents an internal Kripke frame.
In this Kripke frame, each possible world is $\jsh{X}$ for $j\in\PosetP$, the restriction maps are the units $\eta_{X}^{j, k}\colon L_{j}X \to L_{k}X$, and the logical connectives and the quantifiers are interpreted in $\SubtoposjE$: 
\[(\PosetP, \, \leq_{\PosetP}, \, \{\, \jsh{X} \mid j\in\PosetP \,\}, \, \{\, \eta_{X}^{j, k}\colon L_{j}X \to L_{k}X \mid j\leq k \,\}).\]

\begin{definition}\label{def:LforceinP}
    For any $\langX$-formula $\varphi [\vec{x}]$, 
    we inductively define $j\LforceinP \varphi [\vec{u}]$ as follows: 
    \[\begin{aligned}
        j\LforceinP (R[\vec{x}]) &\mycoloneqq \forall \vec{x}:X. ((\bigwedge_{1\leq i\leq m}(x_{i}\in_{X} u_{i})) \rightarrow jR[\vec{x}]); \\
        j\LforceinP (\varphi\land \psi) &\mycoloneqq (j\LforceinP \varphi) \land (j\LforceinP \psi); \\
        j\LforceinP (\varphi\lor \psi) &\mycoloneqq j((j\LforceinP \varphi) \lor (j\LforceinP \psi)); \\
        j\LforceinP (\varphi\rightarrow \psi [\vec{x}]) &\mycoloneqq \\ 
         \forall k\geqinP{j}. &((k\LforceinP \varphi [\intunitX{k}{\vec{u}}] ) \rightarrow (k\LforceinP \psi [\intunitX{k}{\vec{u}}] )); \\
        j\LforceinP (\exists y:X.\varphi[\vec{x}, y]) &\mycoloneqq j(\exists v\in\intLX{j}. j\LforceinP \varphi[\vec{u}, v]); \\
        j\LforceinP (\forall y:X.\varphi[\vec{x}, y]) &\mycoloneqq \forall k \geqinP{j} \forall v\in\intLX{k}. k\LforceinP \varphi[\intunitX{k}{\vec{u}}, v],
    \end{aligned}\]
    where $[\intunitX{k}{\vec{u}}]$ denotes $[\intunitX{k}{u_1}, \dots, \intunitX{k}{u_{m}} ]$:
    \[\varphi:\Omega \quad[\vec{x}:X] \quad\mapsto\quad j\LforceinP \varphi:\Omega \quad [\PosetP:P(P\Omega), j:P\Omega, \vec{u}:PX]. \qedhere\]
\end{definition}

\vspace{\baselineskip} 

As expected, this translation is equivalent to our original translation $j\forceinP\varphi$.
\begin{proposition}\label{prop:LforceinPvsforceinP} 
    For any $\langX$-formula $\varphi [\vec{x}]$, the following equivalence holds: 
    \[\begin{aligned}
        \CatE\models \forall &\PosetP\subLopE \forall j\inLopE \\
        &\forall \vec{x}:X.  ( j\LforceinP\varphi[\intunitX{j}{\singletonX{x_{1}}}, \dots, \intunitX{j}{\singletonX{x_{m}}}] \leftrightarrow j\forceinP\varphi[\vec{x}] ).
    \end{aligned}\]
\end{proposition}
\begin{proof}
    By induction on the complexity of $\varphi$.
    The atomic case follows from Lemma~\ref{lem:loplem}.
    The induction steps for $\land$ and $\lor$ is trivial.
    Let us show the step for existential quantification $\exists y:X. \varphi[x, y]$.
    Using Lemma~\ref{lem:loplem} (4), we have the following equivalence: 
   \[\begin{aligned}
        &j\forceinP(\exists y:X. \varphi) [ \intunitX{j}{\singletonX{x}} ] \\
        = &j(\exists v:PX. j( \exists y:X. v=_{PX} \intunitX{j}{\singletonX{y}} ) \land j\LforceinP \varphi[\intunitX{j}{\singletonX{x}}, v]) \\
        \leftrightarrow &j(\exists v:PX. ( \exists y:X. v=_{PX} \intunitX{j}{\singletonX{y}} ) \land j\LforceinP \varphi[\intunitX{j}{\singletonX{x}}, v] ) \\
        \leftrightarrow &j(\exists y:X. j\LforceinP \varphi[\intunitX{j}{\singletonX{x}}, \intunitX{j}{\singletonX{y}}] ).
   \end{aligned}\]
   By the induction hypothesis, the last formula is equivalent to $j\forceinP (\exists y:X. \varphi[x, y])$.
    
   Next, consider the step for universal quantification $\forall y:X. \varphi[x, y]$.
   Note that $j\LforceinP\varphi$ is $j$-closed similarly to Lemma~\ref{lem:jclosed}.
   Then, by Lemma~\ref{lem:intunit} (2) and the induction hypothesis, 
   we obtain the following equivalence:
   \[\begin{aligned}
        &j\LforceinP(\forall y:X. \varphi) [ \intunitX{j}{\singletonX{x}} ] \\
        = &\forall k\geqinP{j} \forall v\in\intLX{k}. k\LforceinP\varphi [ \intunitX{k}{\intunitX{j}{\singletonX{x}}}, v ] \\
        \leftrightarrow &\forall k\geqinP{j} \forall y:X. k\LforceinP\varphi [\intunitX{k}{\singletonX{x}}, \intunitX{k}{\singletonX{y}}] \\
        \leftrightarrow &\forall k\geqinP{j} \forall y:X. k\LforceinP\varphi[x, y] \quad = \quad j\forceinP(\forall y:X. \varphi) [x]. 
   \end{aligned}\] 
   The step for implication $\rightarrow$ can be shown similarly.
\end{proof}

Using this proposition, we can give an alternative proof of Proposition~\ref{prop:shvsext} based on the internal logic.
\begin{proof}[proof of Proposition~\ref{prop:shvsext}]
    Let $j\colon \Omega \to \Omega$ be a local operator, and consider the singleton lop-frame $\PosetP \mycoloneqq \{\, \transpose{j} \,\}_{P\Omega}$.
    Similarly to Proposition~\ref{prop:maximal}, we can verify the following equivalence by induction on the complexity of $\varphi$:
    \[\forall \vec{r}:\jsh{X}. ( \shjtrans{\varphi}[\vec{r}] \leftrightarrow j\LforceinP\varphi [m(r_{1}), \dots, m(r_{m})] ), \] 
    where $m\colon \jsh{X} \rightarrowtail PX$ (for the atomic case, note Lemma~\ref{lem:pullback} (1)).
    Letting ${r}_{i} \mycoloneqq \eta_{X}({x}_{i})$, we have $m(\eta_{X}(x_{i})) = \intunitX{\transpose{j}}{\singletonX{x_{i}}}$.
    Then, by Proposition~\ref{prop:LforceinPvsforceinP}, we obtain: 
    \[\begin{aligned}
        \shjtrans{\varphi}[\eta_{X}(x_{1}), \dots, \eta_{X}(x_{m})] 
        &\leftrightarrow \transpose{j}\LforceinP\varphi [m(\eta_{X}(x_{1})), \dots, m(\eta_{X}(x_{m}))] \\
        &\leftrightarrow \transpose{j}\LforceinP\varphi [\intunitX{\transpose{j}}{\singletonX{x_{1}}}, \dots, \intunitX{\transpose{j}}{\singletonX{x_{m}}}] \\
        &\leftrightarrow \transpose{j}\forceinP\varphi[\vec{x}].
    \end{aligned}\]
    By Proposition~\ref{prop:maximal}, the last formula is equivalent to $\extjtrans{\varphi}[\vec{x}]$.
\end{proof}

\subsection{Soundness}\label{subsec:soundnessIQC}
We now establish the soundness theorem of our translation for $\IQC$.
First, we observe that our translation satisfies monotonicity, a fundamental property of Kripke models for intuitionistic logic.
It is worth noting that the results in this section do not require the assumption $j\inPosetP$, except for Corollary~\ref{cor:jinPmonotonicity} and \ref{cor:jinPconstantdomain}.

\begin{lemma}[Monotonicity]\label{lem:monotonicity}
    For any $\langX$-formula $\varphi[\vec{x}]$, 
    \[\begin{aligned}
        \CatE\models \forall&\PosetP\subLopE \forall j\inLopE \forall k\geqinP{j} \\
        &\forall \vec{x}:X. (j\forceinP\varphi[\vec{x}] \rightarrow k\forceinP\varphi[\vec{x}]).
    \end{aligned}\]
\end{lemma}
\begin{proof}
    This claim is easily verified by induction on the complexity of $\varphi$.
\end{proof}

Furthermore, if $j\inPosetP$, monotonicity can be expressed by the following equivalence.
\begin{corollary}\label{cor:jinPmonotonicity} 
    For any $\langX$-formula $\varphi[\vec{x}]$, 
    \[\begin{aligned}
        \CatE\models \forall&\PosetP\subLopE \forall j\inPosetP \\
        &\forall \vec{x}:X. (j\forceinP\varphi[\vec{x}] \leftrightarrow (\forall k\geqinP{j}. k\forceinP\varphi[\vec{x}])).
    \end{aligned}\]
\end{corollary}

This corollary implies that our translation behaves analogously to Kripke frame with constant domain. 
\begin{corollary}\label{cor:jinPconstantdomain} 
    For any $\langX$-formula $\varphi[\vec{x}]$, the following equivalence holds:
    \[\begin{aligned}
        \CatE\models \forall&\PosetP\subLopE \forall j\inPosetP. \\
        &(\forall \vec{x}:X. j\forceinP \varphi[\vec{x}]) \leftrightarrow j\forceinP (\forall \vec{x}:X. \varphi).
    \end{aligned}\]
\end{corollary}
\begin{proof}
    This follows from repeated applications of Corollary~\ref{cor:jinPmonotonicity} and Lemma~\ref{lem:jclosed}.
\end{proof}

Consequently, assuming $j\inPosetP$, the clause for universal quantification in Definition~\ref{def:forceinP} admits a simpler formulation given above.
With these preparations, we prove the soundness of our translation for the intuitionistic predicate calculus $\IQC$.
Here, we adopt a Hilbert-style axiomatization, often referred to as {G\"odel's system} \cite[Chapter~3]{Kohlenbachbook}.
\begin{theorem}[Soundness for $\IQC$]\label{thm:soundness}
    Let $\varphi$, $\varphi_1$, $\varphi_2$, and $\psi$ be $\langX$-formulas. 
    \begin{enumerate}
        \item If $\varphi[\vec{x}]$ is an axiom of $\IQC$, then 
        \[\CatE\models \forall\PosetP\subLopE \forall j\inLopE \forall \vec{x}:X. j\forceinP \varphi[\vec{x}].\]

        \item If $\infer{\psi[\vec{x}]}{\varphi_1[\vec{x}] & \varphi_2[\vec{x}]}$ is an inference rule for logical connectives of $\IQC$, then 
        \[\begin{aligned}
            \CatE\models &\forall\PosetP\subLopE \forall j\inLopE \\
            &\forall \vec{x}:X. ((j\forceinP \varphi_1[\vec{x}] \land j\forceinP \varphi_2[\vec{x}]) \rightarrow j\forceinP \psi[\vec{x}]).
        \end{aligned}\]

        \item If $\infer{\psi[\vec{x}]}{\varphi[\vec{x}, y]}$ is an inference rule for quantifiers of $\IQC$, then 
        \[\begin{aligned}
            \CatE\models &\forall\PosetP\subLopE \forall j\inLopE \\
            &\forall \vec{x}:X. ((\forall y:X. j\forceinP \varphi[\vec{x}, y]) \rightarrow j\forceinP \psi[\vec{x}]).
        \end{aligned}\]
    \end{enumerate}
\end{theorem}
\begin{proof}
    (1) We first verify the axioms. 

    \myaxiom{$\varphi\rightarrow \varphi\lor\psi$}
    The translation $j\forceinP (\varphi\rightarrow \varphi\lor\psi)$ is expressed as: 
    \[\forall k\geqinP{j}. ((k\forceinP \varphi) \rightarrow k(k\forceinP \varphi \lor k\forceinP \psi)).\]
    Since $\forall p, q:\Omega. (p \rightarrow k(p\lor q))$ holds for any $k\inLopE$, 
    the above formula is valid in $\CatE$.
    Similarly, one can easily verify the cases for $\varphi\land\psi \rightarrow \varphi$, the axioms for contraction 
    $\varphi\lor\varphi \rightarrow \varphi$ and $\varphi \rightarrow \varphi\land\varphi$, 
    and the axioms for exchange 
    $\varphi\lor\psi \rightarrow \psi\lor\varphi$ and $\varphi\land\psi \rightarrow \psi\land\varphi$.

    \myaxiom{$\bot\rightarrow \varphi$} 
    By $(k\forceinP \bot) \leftrightarrow k\bot$ and Lemma~\ref{lem:jclosed}, 
    we have the following equivalence: 
    \[\begin{aligned}
        &j\forceinP (\bot\rightarrow\varphi) \\
        \leftrightarrow &\forall k\geqinP{j}. (k\bot\rightarrow k\forceinP \varphi) \\
        \leftrightarrow &\forall k\geqinP{j}. (\bot\rightarrow k\forceinP \varphi).
    \end{aligned}\]
    The last formula follows from ex falso quodlibet in the internal logic of $\CatE$.

    \myaxiom{$\forall x:X. \varphi[x]\rightarrow \varphi[t[\vec{y}]]$}
    Note that the translation $j\forceinP\varphi[{x}]$ commutes with substitution,
    that is, $j\forceinP(\varphi[t])$ coincides with $j\forceinP(\varphi[x]) [t/x]$.
    By Corollary~\ref{cor:jinPconstantdomain}, we obtain the following equivalence:
    \[\begin{aligned}
        &j\forceinP(\forall x:X. \varphi[x]\rightarrow \varphi[t[y]]) \\
        \leftrightarrow &\forall k\geqinP{j}. (k\forceinP(\forall x:X. \varphi[x]) \rightarrow k\forceinP(\varphi[t[y]])) \\
        \leftrightarrow &\forall k\geqinP{j}. ( (\forall x:X. k\forceinP\varphi[x]) \rightarrow k\forceinP\varphi [t[y]/x]).
    \end{aligned}\]
    The last formula is clearly valid in $\CatE$.
    The axiom $(\varphi[t[\vec{y}]] \rightarrow \exists x:X. \varphi[x])$ is verified similarly.

    (2) Next, we verify the inference rules for logical connectives.

    \myinfrule{$\infer={\varphi \rightarrow (\psi \rightarrow \chi)}{\varphi\land\psi \rightarrow \chi}$}
    By definition, $j\forceinP \varphi\land\psi \rightarrow \chi$ is expressed as:
    \[\forall k\geqinP{j}. ((k\forceinP\varphi) \land (k\forceinP\psi) \rightarrow k\forceinP\chi).\]
    From Lemma~\ref{lem:monotonicity}, $k\forceinP\varphi$ is equivalent to $\forall \ell\geqinP{k}. (\ell\forceinP\varphi)$.
    Hence, the above formula is equivalent to:
    \[\forall k\geqinP{j} \forall \ell\geqinP{k}. ((k\forceinP\varphi \land \ell\forceinP\psi) \rightarrow \ell\forceinP\chi).\]
    This formula is intuitionistically equivalent to: 
    \[\forall k\geqinP{j}. (k\forceinP\varphi \rightarrow \forall \ell\geqinP{k}. (\ell\forceinP\psi \rightarrow \ell\forceinP\chi)),\]
    which is coincides with $j\forceinP \varphi \rightarrow (\psi \rightarrow \chi)$.
    Similarly, the rules for modus ponens and cut are verified:  
    \[\infer{\psi}{\varphi & \varphi\rightarrow\psi} \qquad 
    \infer{\varphi \rightarrow \chi}{\varphi\rightarrow\psi & \psi\rightarrow\chi}.\] 

    \myinfrule{$\infer{\varphi\lor\chi \rightarrow \psi\lor\chi}{\varphi\rightarrow\psi}$}
    Noting that $\forall p,q:\Omega. (p\rightarrow q) \rightarrow (kp\rightarrow kq)$ holds for any $k\inLopE$, 
    we can verify the following implication: 
    \[\begin{aligned}
        &j\forceinP (\varphi\rightarrow\psi) \\
        = &\forall k\geqinP{j}. (k\forceinP\varphi \rightarrow k\forceinP\psi) \\
        \rightarrow &\forall k\geqinP{j}. ((k\forceinP\varphi \lor k\forceinP\chi) \rightarrow (k\forceinP\psi \lor k\forceinP\chi)) \\
        \rightarrow &\forall k\geqinP{j}. (k(k\forceinP\varphi \lor k\forceinP\chi) \rightarrow k(k\forceinP\psi \lor k\forceinP\chi)) \\
        = &\forall k\geqinP{j}. (k\forceinP(\varphi\lor\chi) \rightarrow k\forceinP(\psi\lor\chi)).
    \end{aligned}\]
    The last formula is precisely the definition of $j\forceinP (\varphi\lor\chi \rightarrow \psi\lor\chi)$.

    (3) Finally, we verify the inference rules for quantifiers.
    For simplicity, assume that $\psi$ is an $\langX$-sentence and $\varphi[y]$ is an $\langX$-formula where $y$ is the only free variable.
    Note that the implications in (3) are, in fact, equivalences.

    \myinfrule{$\infer{\psi\rightarrow \forall y:X. \varphi[y]}{\psi\rightarrow\varphi[y]}$}
    By Corollary~\ref{cor:jinPconstantdomain}, we obtain the following equivalence: 
    \[\begin{aligned}
        &\forall y:X. j\forceinP (\psi\rightarrow \varphi[y]) \\
        \myeqq &\forall y:X \forall k\geqinP{j}. (k\forceinP\psi \rightarrow k\forceinP \varphi[y]) \\
        \leftrightarrow &\forall k\geqinP{j}. (k\forceinP\psi \rightarrow \forall y:X. k\forceinP \varphi[y]) \\
        \leftrightarrow &\forall k\geqinP{j}. (k\forceinP\psi \rightarrow k\forceinP (\forall y:X. \varphi)).
    \end{aligned}\]
    The last formula is nothing but $j\forceinP (\psi\rightarrow \forall y:X. \varphi)$.

    \myinfrule{$\infer{\exists y:X. \varphi[y]\rightarrow \psi}{\varphi[y]\rightarrow\psi}$}
    by applying Lemma~\ref{lem:jclosed}, we have the following equivalence: 
    \[\begin{aligned}
        &\forall y:X. j\forceinP (\varphi[y]\rightarrow\psi) \\
        \myeqq &\forall y:X \forall k\geqinP{j}. (k\forceinP\varphi[y] \rightarrow k\forceinP\psi) \\
        \leftrightarrow &\forall k\geqinP{j}. ((\exists y:X. k\forceinP \varphi[y] ) \rightarrow k\forceinP\psi) \\
        \leftrightarrow &\forall k\geqinP{j}. (k(\exists y:X. k\forceinP \varphi[y] ) \rightarrow k\forceinP\psi) \\
        \myeqq &\forall k\geqinP{j}. (k\forceinP (\exists y:X. \varphi[y]) \rightarrow k\forceinP\psi) 
    \end{aligned}\]
    The last formula coincides with the definition of $j\forceinP(\exists y:X. \varphi [y] \rightarrow \psi)$.
    This completes the proof. 
\end{proof}

\subsection{Heyting arithmetic}\label{subsec:soundnessHA}
We show that our translation is sound for Heyting arithmetic $\HA$ under the standard interpretation on a natural numbers object $N$.
The axioms to be verified consist not only of the induction axiom scheme but also of the universal closures of literals, such as the successor axiom and defining equations for primitive recursive functions.
To this end, we first consider a subclass of formulas which includes the latter.

\begin{definition}\label{def:literalclass}
    Let $\myfmlcls$ denote the subclass of $\langX$-formulas generated by the following rules: 
    \[\varphi, \psi \Coloneqq R[\vec{x}] \mid \lnot R[\vec{x}] \mid \varphi\land\psi \mid \forall x:X. \varphi. \qedhere\]
\end{definition}

Then the following property holds for any formula in $\myfmlcls$. 

\begin{proposition}\label{prop:literal}
    For any $\langX$-formula $\varphi[\vec{x}]$ in $\myfmlcls$, 
    \[\CatE\models \forall \vec{x}:X. (\varphi[\vec{x}] \leftrightarrow \forall \PosetP\subLopE \forall j\inLopE. j\forceinP \varphi[\vec{x}] ).\]
\end{proposition}
\begin{proof}
    For the backward direction, consider the identity local operator $\mathrm{id}\colon\Omega\to\Omega$ and the singleton lop-frame $\PosetP\mycoloneqq \{\, \transpose{\mathrm{id}} \,\}_{P\Omega}$.
    The claim then follows from Proposition~\ref{prop:maximal}.

    The forward direction is proved by induction on the complexity of $\varphi$.
    Here, we verify only the case of $\lnot R[x]$.
    By definition, the translation $j\forceinP (\lnot R[x])$ is expressed as $\forall k\geqinP{j}. (kR[x] \rightarrow k\bot)$.
    This formula is equivalent to $\forall k\geqinP{j}. (R[x] \rightarrow k\bot)$, which follows from $\lnot R[x]$.
\end{proof}

Thus, it suffices to show that the induction axiom is forced. 
A key fact for this proof is that induction over the natural numbers object holds for arbitrary formula with parameters (the proof is essentially contained in \cite[Theorem~4.1]{LambekScott}).
\begin{theorem}[Soundness for $\HA$]\label{thm:soundnessHA}
    Let $\CatE$ be an elementary topos with a natural numbers object $(N, 0, \Sucs)$.
    Let $\varphi$ be an $\Arithl$-formula.
    If $\varphi$ is an axiom of $\HA$, then the following holds under the standard interpretation on $N$:
    \[\CatE\models \forall\PosetP\subLopE \forall j\inLopE. j\forceinP\varphi.\]
\end{theorem}
\begin{proof}
    If $\varphi$ is the successor axiom $\forall x:N. \lnot(\Sucm(x)=_{N}0)$ or a defining equation $\forall \vec{x}:N. (f(\vec{x}) =_{N} g(\vec{x}))$ for a primitive recursive function, 
    the claim follows from Proposition~\ref{prop:literal} and the fact that $\varphi$ holds in $\CatE$.
    It remains to show that the induction axiom scheme: 
    \[(\mathrm{Ind}_{\psi}) \mycoloneqq \psi[0] \land \forall x:N. (\psi[x]\rightarrow \psi[\Sucm(x)]) \rightarrow \forall x:N.\psi[x].\]
    Let $k\inLopE$ satisfy $j\leq k$ and $k\inPosetP$.
    Since $k\inPosetP$, Corollary~\ref{cor:jinPmonotonicity} implies that 
    $k\forceinP(\psi[0])$ is equivalent to $\forall \ell\geqinP{k}. \ell\forceinP\psi [0]$.
    Similarly, $k\forceinP (\forall x:N. (\psi[x]\rightarrow \psi[\Sucm(x)]))$ is equivalent to: 
    \[\forall \ell\geqinP{k} \forall x:N. ( \ell\forceinP\psi[x] \rightarrow \ell\forceinP \psi[\Sucm(x)] ).\]

    Since $N$ is a natural numbers object, the induction holds for the formula $\ell\forceinP\psi[x]$.
    Therefore, we obtain $\forall \ell\geqinP{k} \forall x:X. (\ell\forceinP\psi[x])$ from $k\forceinP(\psi[0])$ and $k\forceinP (\forall x:N. (\psi[x]\rightarrow \psi[\Sucm(x)]))$.
    This means the following implication: 
    \[\begin{aligned}
        \forall k\geqinP{j}. &(k\forceinP(\psi[0]) \land k\forceinP (\forall x:N. (\psi[x]\rightarrow \psi[\Sucm(x)])) \\
        &\rightarrow k\forceinP(\forall x:N. \psi[x]) ), 
    \end{aligned}\]
    which is precisely the definition of $j\forceinP (\mathrm{Ind}_{\psi})$. 
\end{proof}

\subsection{Kuroda-style translation}\label{subsec:Kuroda} 
We conclude this section by pointing out that our translation admits an alternative definition based on the \emph{Kuroda-style $j$-translation}.
In intuitionistic proof theory, there are several variations of double negation translation, such as the G\"{o}del-Gentzen-style, the Kuroda-style, and the Kolmogorov-style (see, for example, \cite[Chapter~2~Section~3]{ConstructivismvolI}).
A similar situation holds for the $j$-translation.
In particular, \cite{vdBerg19} provides a syntactic study of the Kuroda-style $j$-translation.

While Definition~\ref{def:forceinP} is based on the G\"{o}del-Gentzen-style, 
one can define a variant $j\KforceinP\varphi$ corresponding to the Kuroda-style $j$-translation: 
\[\begin{aligned}
    j\KforceinP (R[\vec{x}]) &\mycoloneqq R[\vec{x}]; \\
    j\KforceinP (\varphi\land \psi) &\mycoloneqq (j\KforceinP \varphi) \land (j\KforceinP \psi); \\
    j\KforceinP (\varphi\lor \psi) &\mycoloneqq (j\KforceinP \varphi) \lor (j\KforceinP \psi); \\
    j\KforceinP (\varphi\rightarrow \psi [\vec{x}]) &\mycoloneqq \\
    \forall k\geqinP{j}. & ((k\KforceinP \varphi [\vec{x}]) \rightarrow k(k\KforceinP \psi [\vec{x}])); \\
    j\KforceinP (\exists y:X.\varphi[\vec{x}, y]) &\mycoloneqq \exists y:X. j\KforceinP \varphi[\vec{x}, y]; \\
    j\KforceinP (\forall y:X. \varphi[\vec{x}, y]) &\mycoloneqq \forall k \geqinP{j} \forall y:X. k(k\KforceinP \varphi[\vec{x}, y]).
\end{aligned}\]

The Kuroda-style translation is then given by $j(j\KforceinP\varphi)$.
As is the case for the usual $j$-translation, this is indeed equivalent to the original translation.
\begin{proposition}\label{prop:KvsGG}
    For any $\langX$-formula $\varphi[\vec{x}]$, 
    \[\begin{aligned}
        \CatE\models &\forall\PosetP\subLopE \forall j\inLopE \\
        &\forall \vec{x}:X. (j(j\KforceinP\varphi [\vec{x}]) \leftrightarrow j\forceinP \varphi[\vec{x}]).
    \end{aligned}\]
\end{proposition}
\begin{proof}
    By induction on the complexity of $\varphi$.
    It is useful to use Lemma~\ref{lem:loplem} and Lemma~\ref{lem:jclosed}.
\end{proof}

\subsection{A corresponding internal sheaf topos}\label{subsec:internalsheaftopos} 
Benno van den Berg pointed out the existence of a topos corresponding to our translation (private communication).
Here, we sketch a candidate for the construction based on discussions with him.

We work in an elementary topos $\CatE$.
Let $\PosetP\subLopE$ be a lop-frame, and assume that $j\inPosetP$.
Then the opposite poset $\opPosetP$ can be regarded as an internal category in $\CatE$, 
and we can define the internal presheaf topos $\PreshtopP$ for $\opPosetP$ in $\CatE$.
We consider the following:
\[\begin{aligned}
    \rest{\PosetP}{j} &\mycoloneqq \{\, k\in_{} \PosetP \mid j\leq k \,\}; \\
    \subobclsP(j) &\mycoloneqq \{\, U\subseteq \rest{\PosetP}{j} \mid \forall k\in U \forall \ell\geqinP{k}. (\ell\in_{} U) \,\}; \\
    \GrothcovP(j) &\mycoloneqq \{\, U \in_{} \subobclsP(j) \mid \forall k\geqinP{j}. j(k\in_{} U) \,\}.
\end{aligned}\]
Here, $\subobclsP(j)$ represents the family of all sieves on $j$ in $\opPosetP$, 
and defines the subobject classifier of $\PreshtopP$.
Note that $\rest{\PosetP}{j}$ is the maximal sieve on $j$.
In this setting, we can verify that $(\opPosetP, \GrothcovP)$ forms an internal site in $\CatE$.
This defines the internal sheaf topos $\ShtopP$ in $\CatE$ (for details on internal sites and internal sheaf toposes, see \cite[Section~C2.4]{ElephantVol1}).

It can be observed that the internal logic of this topos $\ShtopP$ corresponds to our translation.
First, for $U \in_{} \subobclsP(j)$, we define $\lopP(j) (U)$ as follows: 
\[\lopP(j) (U) \mycoloneqq \{\, k\in_{} \rest{\PosetP}{j} \mid j(k\in_{} U) \,\}.\]
This defines a local operator $\lopP\colon \subobclsP \to \subobclsP$ corresponding to the Grothendieck coverage $\GrothcovP$.
(Indeed, for any $U\in_{} \subobclsP(j)$, $U \in_{}\GrothcovP(j)$ if and only if $\lopP(j) (U) = \rest{\PosetP}{j}$.)
Now, let us look at the structure of the lattice of subterminals in $\ShtopP$.
Each subterminal $U \rightarrowtail 1$ in $\ShtopP$ is determined by an upset $U$ of $\PosetP$ such that $\lopP(j)(\rest{U}{j}) = \rest{U}{j}$ holds for any $j\inPosetP$, 
where $\rest{U}{j}$ is defined similarly to $\rest{\PosetP}{j}$.
Therefore, for subterminals $U$, $V\rightarrowtail 1$, the join $\lor_{\GrothcovP}$ and the Heyting implication $\Rightarrow_{\GrothcovP}$ in $\ShtopP$ 
are calculated using those $\lor$ and $\Rightarrow$ in $\PreshtopP$:
for any $j\in\PosetP$, 
\[\begin{aligned}
    \rest{(U \lor_{\GrothcovP} V)}{j} 
    &= \lopP(j) (\rest{U}{j} \lor \rest{V}{j}) \\
    &= \{\, \ell\in_{} \rest{\PosetP}{j} \mid j( \ell \in_{} \rest{U}{j} \lor \ell \in_{} \rest{V}{j} )\,\} \\
    \rest{(U \Rightarrow_{\GrothcovP} V)}{j} 
    &= (\rest{U}{j} \Rightarrow \rest{V}{j}) \\
    &= \{\, \ell\in_{} \rest{\PosetP}{j} \mid \forall k\geqinP{j}. ( \ell \in_{} \rest{U}{k} \rightarrow \ell \in_{} \rest{V}{k} )\,\} 
\end{aligned}\]
These coincide with the clauses for disjunction and implication in Definition~\ref{def:lopframe}, respectively.

In Section~\ref{subsec:shjtrans}, we saw that the usual $j$-translation $\extjtrans{\varphi}$ is exactly related to the internal logic of the subtopos $\SubtoposjE$ via the $\jshfunc$-translation (Proposition~\ref{prop:shvsext}).
Similarly, we expect that any $\langX$-sentence $\varphi$ induces a subterminal $\subinp{\varphi}_{\GrothcovP} \rightarrowtail 1$ in $\ShtopP$, 
and that $j\forceinP\varphi$ in $\CatE$ describes $\rest{\subinp{\varphi}_{\GrothcovP}}{j}$ in $\ShtopP$.
Through this correspondence, categorical observations on $\ShtopP$ could also yield the results established in this section, such as the soundness for Heyting arithmetic.
However, we leave further investigation of this topos for future work.
We will briefly discuss future problems concerning our translation and its relation to this topos in Section~\ref{sec:conclusion}.

\section{The corresponding realizability semantics}\label{sec:realizability}

In this section, we interpret the formula $j\forceinP\varphi$ in the effective topos $\Eff$.
Recall that a formula $\varphi[x]$ in the internal language of a topos induces a subobject $\subinp{\varphi}\rightarrowtail X$.
In $\Eff$, such a subobject can be described as a $\powN$-valued function $\intinEff{\varphi}\colon \undset{X} \to \powN$ on a set $\undset{X}$.
Our first goal is to provide an explicit and concise description of the $\powN$-valued function $\intinEff{j\forceinP\varphi}$ for an $\Arithl$-formula $\varphi$. 

To this end, in Section~\ref{subsec:preliminariesEff}, we review the basics of the effective topos.
Especially, we explain the structure of power objects in details to clarify the internal object $\LopFE$ of all lop-frames.
In Section~\ref{subsec:Preal}, we provide a concrete description of $\intinEff{j\forceinP\varphi}$, which we call $\PosetP$-realizability.
We then elucidate the properties of $\PosetP$-realizability when the lop-frame $\PosetP$ is standard.
Finally, in Section~\ref{subsec:deJonghGoodman}, based on a uniform construction for local operators, 
we show that de Jongh-Goodman realizability is subsumed by our $\PosetP$-realizability.

It should be noted that most of the arguments below can be easily extended to the relative realizability topos over a relative partial combinatory algebra.

\subsection{Preliminaries on $\Eff$}\label{subsec:preliminariesEff}
We begin by introducing the notation for the effective tripos $X \mapsto (\powN^X, \leqEff)$.
We refer readers interested in the basics of the effective topos to the standard text \cite{vOostenbook}.

\begin{definition}\label{def:intlanginEff}
    We fix a primitive recursive pairing function $\langle -, -\rangle\colon \N^{2}\to \N$.
    For natural numbers $e$ and $n$, we write $e\Klapp n$ for the result of applying the $e$-th partial computable function to $n$, 
    and write $e\Klapp n \downarrow$ if the computation terminates.
    A $\powN$-valued function $\Phi\colon X\to \powN$ on a set $X$ is called a \emph{predicate (on $X$)}. 
    Given a set $X$, we define the following logical operations on the set of predicates on $X$:
    for any $\Phi(x)$, $\Psi(x)\colon X\to \powN$, 
    \[\begin{aligned}
        \Phi\land\Psi (x) &\mycoloneqq \{\, \langle n, m\rangle \mid n\in \Phi(x) \text{ and } m\in \Psi(x) \,\}; \\
        \Phi\lor\Psi (x) &\mycoloneqq \{\, \langle 0, n\rangle \mid n\in \Phi(x) \,\} \cup \{\, \langle 1, m\rangle \mid m\in \Psi(x) \,\}; \\
        \Phi\rightarrow\Psi (x) &\mycoloneqq \{\, e \mid \forall n\in \Phi(x). (e\Klapp n\downarrow \text{ and } e\Klapp n\in \Psi(x)) \,\}.
    \end{aligned}\]
    We write $\top(x)\mycoloneqq \N$, $\bot(x)\mycoloneqq \emptyset$, $\lnot\Phi \mycoloneqq \Phi\rightarrow \bot$, and
    $(\Phi\leftrightarrow\Psi) \mycoloneqq (\Phi\rightarrow\Psi) \land (\Psi\rightarrow\Phi)$.
    The same operations are introduced on $\powN$. 

    A \emph{realizer of $\Phi(x)$} is a natural number contained in $\Phi(x)$.
    We say that \emph{$\Phi(x)$ is realizable} if $\Phi(x)\neq\emptyset$.
    We define the preorder predicate $(\Phi \leqEff \Psi) \mycoloneqq \bigcap_{x\in X} (\Phi \rightarrow \Psi) (x)$, 
    and the equivalence predicate $(\Phi \eqEff \Psi) \mycoloneqq \bigcap_{x\in X} (\Phi \leftrightarrow \Psi) (x)$ between predicates on $X$.
    If $(\Phi \eqEff \Psi)$ (resp. $(\Phi \leqEff \Psi)$) is realizable, we simply write $\Phi \eqEff \Psi$ (resp. $\Phi \leqEff \Psi$).
    We say that \emph{$\Phi$ and $\Psi$ are equivalent} if $\Phi \eqEff \Psi$.
\end{definition}

We next recall the objects in $\Eff$.
An object $X$ in $\Eff$ is a pair $(\undset{X}, =_{X})$ of a set $\undset{X}$ and a partial equivalence relation $=_{X}\colon \undset{X}\times \undset{X} \to \powN$
(see \cite[Section~3.1]{vOostenbook}).
For instance, the subobject classifier is given by $\Omega \mycoloneqq (\OmegaEff, =_{\Omega})$ with $(p=_{\Omega} q) \mycoloneqq (p\leftrightarrow q)$.
The natural numbers object is defined as $N \mycoloneqq (\N, =_{N})$, where $(n =_{N} m) \mycoloneqq \{\, n\,\}$ if $n=m$ and $\mycoloneqq \emptyset$ otherwise.

Let us clarify the descriptions of an internal object $\LopE$ of local operators and an internal object $\LopFE$ of lop-frames in $\Eff$.
To do so, we briefly recall the construction of power objects in $\Eff$.
The power object $PX$ of an object $X = (\undset{X}, =_{X})$ in $\Eff$ is defined by the object $({\OmegaEff}^{\undset{X}}, =_{PX})$ of strict relational predicates with respect to $=_{X}$
(see \cite[page.~70]{vOostenbook}).
Furthermore, if $X$ is a uniform object (such as $\Omega$ or $P\Omega$), the strictness condition can be omitted (\cite[Proposition~3.2.6]{vOostenbook}).
\begin{definition}\label{def:PXinEff}
    Let $X = (\undset{X}, =_{X})$ be an object in $\Eff$.
    We introduce the predicate $\Rel{X}$ on $\powN^{\undset{X}}$ as follows: 
    \[\Rel{X}(\Phi) \mycoloneqq \bigcap_{x, y\in \undset{X}} (\Phi(x) \land (x=_{X} y) \rightarrow \Phi(y)).\]
    We say that a predicate $\Phi$ on $\undset{X}$ is \emph{$=_{X}$-relational} if $\Rel{X}(\Phi)$ is realizable.
    Similarly, for $\eqEff$ in Definition~\ref{def:intlanginEff}, 
    we define the predicate $\mathrm{Rel}_{\eqEff}$ and the notion of \emph{$\eqEff$-relationality}.
\end{definition}

\begin{remark}\label{rem:PXinEff}
    The power object of $\Omega$ is defined by $P\Omega \mycoloneqq (\OmegaEff^{\OmegaEff}, =_{P\Omega})$, where $(j =_{P\Omega} k) \mycoloneqq \Rel{\Omega}(j) \land (j\eqEff k)$.
    This implies a one-to-one correspondence (up to equivalence) between global sections of $P\Omega$ and $=_{\Omega}$-relational predicates on $\OmegaEff$.
    Similarly, the power object of $P\Omega$ is defined by $P(P\Omega) \mycoloneqq (\OmegaEff^{(\OmegaEff^{\OmegaEff})}, =_{P(P\Omega)})$, where $(\PosetP =_{P(P\Omega)} \PosetQ) \mycoloneqq \Rel{P\Omega}(\PosetP) \land (\PosetP\eqEff \PosetQ)$.
    Thus, there is a one-to-one correspondence (up to equivalence) between global sections of $P(P\Omega)$ and $=_{P\Omega}$-relational predicates on $\OmegaEff^{\OmegaEff}$.
    The membership relation is interpreted as $\intinEff{j \in_{} \PosetP} \mycoloneqq \PosetP(j)$.
\end{remark}

Let us recall that $\LopE$ and $\LopFE$ were defined as subobjects 
$\subinp{\isLopE[j]}\rightarrowtail P\Omega$ (Remark~\ref{rem:POmegavslop}) and $\subinp{\PosetP\subLopE}\rightarrowtail P(P\Omega)$ (Remark~\ref{rem:POmegavslop}), respectively.
These definitions imply that local operators and lop-frames in $\Eff$ correspond to specific predicates, as below.
\begin{definition}\label{def:LopinEff}
    For predicates $j$, $k\colon \OmegaEff\to\OmegaEff$, we define: 
    \[\begin{aligned}
        \intinEff{\isLopE[j]} &\mycoloneqq \bigcap_{p\subseteq \N} (p \rightarrow j(p)) \land 
        \bigcap_{p\subseteq \N} (jj(p) \rightarrow j(p)) \\
        &\land \bigcap_{p, q\subseteq \N} ((p\rightarrow q) \rightarrow (j(p)\rightarrow j(q))), \\
        (j =_{\LopE} k) &\mycoloneqq \intinEff{\isLopE[j]} \land (j \eqEff k).
    \end{aligned}\]
    A predicate $j$ on $\OmegaEff$ is called a \emph{local operator in the effective topos} if $\intinEff{\isLopE[j]}$ is realizable.
    We write $\LopEff$ for the set of all local operators in $\Eff$.
    The standard order between local operators is interpreted as $\intinEff{j\leq k} \mycoloneqq (j \leqEff k)$.

    A local operator equivalent to the largest element $\top(p)\mycoloneqq \N$ in the preordered set $(\LopEff, \leqEff)$
    is called \emph{degenerate}.
    In addition, the \emph{double negation} operator is interpreted as follows: $\lnot\lnot(p) \mycoloneqq \N$ if $p\neq \emptyset$ and $\mycoloneqq \emptyset$ otherwise.
    We say that $j$ is \emph{dense} if $j\leqEff \lnot\lnot$, and write $\DLopEff$ for the set of all dense local operators in $\Eff$.
\end{definition}

\begin{remark}\label{rem:lopinEff}
    As is well known, the following are equivalent in $\Eff$: 
    (1) $j$ is non-degenerate; 
    (2) $j$ is dense; 
    (3) for any $p\subseteq\N$, $j(p)\neq \emptyset$ if and only if $p\neq \emptyset$.
\end{remark}

\begin{definition}\label{def:LopFinEff}
    For predicates  $\PosetP$, $\PosetQ \colon \OmegaEff^{\OmegaEff} \to \OmegaEff$, we define: 
    \[\begin{aligned}
        \intinEff{\PosetP\subLopE} &\mycoloneqq \bigcap_{j\in \OmegaEff^{\OmegaEff}} (\PosetP(j) \rightarrow \intinEff{\isLopE[j]}), \\
        (\PosetP =_{\LopFE} \PosetQ) &\mycoloneqq \intinEff{\PosetP\subLopE} \land \mathrm{Rel}_{\eqEff}(\PosetP) \land (\PosetP \eqEff \PosetQ).
    \end{aligned}\]
    A predicate $\PosetP$ on $\OmegaEff^{\OmegaEff}$ is called a \emph{lop-frame in the effective topos} if it is $\eqEff$-relational and $\intinEff{\PosetP\subLopE}$ is realizable.
\end{definition}

In summary, the internal objects $\LopE$ and $\LopFE$ is given by the following:
\[\LopE \mycoloneqq (\OmegaEff^{\OmegaEff}, =_{\LopE}), \quad \LopFE \mycoloneqq (\OmegaEff^{(\OmegaEff^{\OmegaEff})}, =_{\LopFE}).\] 

\begin{remark}\label{rem:relationality}
    Since $\intinEff{\isLopE[j]} \leqEff \Rel{\Omega}(j)$ holds, all local operators in $\Eff$ are $=_{\Omega}$-relational.
    Therefore, as noted in Remark~\ref{rem:PXinEff}, each local operator $j$ corresponds to a closed term of type $P\Omega$ that satisfies $\isLopE[j]$ internally.
    This characterization was originally given by Pitts \cite{PittsPhD}.
    Compared to this, each lop-frame $\PosetP$ in $\Eff$ corresponds to a closed term of type $P(P\Omega)$ that satisfies $\PosetP\subLopE$ internally.
\end{remark}

We conclude this subsection by noting that each lop-frame in $\Eff$ can be constructed as an $\eqEff$-relational predicate on $\OmegaEff^{\OmegaEff}$ weighted by $\intinEff{\isLopE[j]}$.
\begin{proposition}\label{prop:weighted} 
    Let $\PosetP\colon \OmegaEff^{\OmegaEff}\to\OmegaEff$ be an $\eqEff$-relational predicate.
    Then the predicate $\Lopclosure{\PosetP}$ defined as follows is a lop-frame in $\Eff$:
    \[\Lopclosure{\PosetP}(j) \mycoloneqq \PosetP(j) \land \intinEff{\isLopE[j]}.\]
    Conversely, any lop-frame in $\Eff$ is equivalent to a predicate of this form.
\end{proposition}
\begin{proof}
    This claim follows from the fact that the predicate $\intinEff{\isLopE[j]}$ is $\eqEff$-relational.
\end{proof}

\subsection{$\PosetP$-realizability}\label{subsec:Preal}
Given a lop-frame $\PosetP$ and a local operator $j$ in $\Eff$, 
the translation $j\forceinP\varphi$ for an $\Arithl$-formula $\varphi[{x}]$ induces a predicate $\Prealat{j}{\varphi}\colon \N\to\powN$.
Based on the Kuroda-style translation, we provide the following description.

\begin{definition}\label{def:Preal}
    Let $\PosetP$ be a lop-frame and $j$ be a local operator in $\Eff$.
    For an $\Arithl$-formula $\varphi[\vec{x}]$,
    we inductively define the predicate $\Prealat{j}{\varphi} \colon \N^{m}\to\powN$ as follows:
    \[\begin{aligned}
        \Prealat{j}{(f(\vec{x}) = g(\vec{x}))} &\mycoloneqq 
            \begin{cases}
                \N & \text{ if } (f(\vec{n}) = g(\vec{n})) \text{ holds in }\N \\
                \emptyset & \text{ otherwise; }
            \end{cases} \\
        \Prealat{j}{(\varphi\land \psi)} &\mycoloneqq \Prealat{j}{\varphi} \land \Prealat{j}{\psi}; \\
        \Prealat{j}{(\varphi\lor \psi)} &\mycoloneqq \Prealat{j}{\varphi} \lor \Prealat{j}{\psi}; \\
        \Prealat{j}{(\varphi\rightarrow \psi)} &\mycoloneqq \\
        \bigcap_{k\in\OmegaEff^{\OmegaEff}} &( \PosetP(k)\land \intinEff{j\leq k} \land \Prealat{k}{\varphi} \rightarrow k(\Prealat{k}{\psi} ) ); \\
        \Prealat{j}{(\exists y.\varphi[\vec{x}, y])} (\vec{n}) &\mycoloneqq \bigcup_{m\in\N} ( \{\, m\,\} \land \Prealat{j}{\varphi} (\vec{n}, m) ); \\
        \Prealat{j}{(\forall y.\varphi[\vec{x}, y])} (\vec{n}) &\mycoloneqq \\
        \bigcap_{k\in\OmegaEff^{\OmegaEff}} \bigcap_{m\in\N} &( \PosetP(k)\land \intinEff{j\leq k} \land \{\, m\,\} \rightarrow k(\Prealat{k}{\varphi} (\vec{n}, m)) ).
    \end{aligned}\]
    For an $\Arithl$-sentence $\varphi$, 
    we say that $\varphi$ is \emph{$\PosetP$-realizable at $j$} if $\Prealat{j}{\varphi}$ is realizable.
\end{definition}

The predicate $\Prealat{j}{\varphi}$ corresponds to the subobject induced by the Kuroda-style translation $j\KforceinP\varphi$ defined in Section~\ref{subsec:Kuroda}.
As noted in Remark~\ref{rem:lopinEff}, if $j$ is non-degenerate, the following equivalence holds externally:
\[\Eff\models j(j\KforceinP\varphi) \iff j(\Prealat{j}{\varphi})\neq\emptyset \iff \Prealat{j}{\varphi}\neq\emptyset.\]
Thus, Proposition~\ref{prop:KvsGG} implies that the validity of the G\"odel-Gentzen style translation $j\forceinP\varphi$ in $\Eff$ coincides with $\PosetP$-realizability at $j$
as long as $j$ is non-degenerate.
\begin{theorem}\label{thm:realsound}
    Let $\PosetP$ be a lop-frame and $j$ be a non-degenerate local operator in $\Eff$.
    Then, for any $\Arithl$-formula $\varphi[\vec{x}]$ and any list $\vec{n}$ of natural numbers, 
    the following equivalence holds:
    \[\Eff\models (j\forceinP(\varphi[\vec{n}])) \iff \Prealat{j}{\varphi} (\vec{n}) \neq \emptyset,\]
    where $\varphi[\vec{n}]$ denotes the sentence obtained by substituting the numerals $\vec{n}$ into $\varphi$.
\end{theorem}

As a corollary of Theorem~\ref{thm:soundness} and \ref{thm:soundnessHA},
$\PosetP$-realizability is sound for Heyting arithmetic.
Furthermore, recalling Corollary~\ref{cor:jinPconstantdomain}, 
we can simplify the clause for universal quantification if $\PosetP(j) = \intinEff{j\in_{} \PosetP}$ is realizable.

\begin{corollary}\label{cor:constinEff}
    Suppose that a lop-frame and a non-degenerate local operator in $\Eff$
    satisfy $\PosetP(j)\neq\emptyset$.
    Then, for any $\Arithl$-formula $\varphi[{y}]$, the following equivalence holds:
    \[\Prealat{j}{(\forall {y}. \varphi[{y}])} \eqEff \bigcap_{m\in\N} (\{\, m\,\} \rightarrow j(\Prealat{j}{\varphi} ({m}))). \]
\end{corollary}

Observe that the usual $j$-translation $\extjtrans{\varphi}$ also corresponds to the predicate $\jreal{\varphi}\colon \N^{m}\to\powN$ defined by:
\[\jreal{(\varphi\rightarrow \psi)} \mycoloneqq \jreal{\varphi} \rightarrow j(\jreal{\psi}), \quad \jreal{(\forall {y}. \varphi[{y}])} \mycoloneqq \bigcap_{m\in\N} (\{\, m\,\} \rightarrow j\jreal{\varphi} ({m})).\]
The other steps are defined analogously to Definition~\ref{def:Preal}.
This predicate is clearly the set of realizers for the \emph{$j$-realizability} found in the literature \cite{Kihara24rethinking, LeevOosten13, vOosten14}.
Henceforth, we also say that $\varphi$ is \emph{$j$-realizable} if $\jreal{\varphi}$ is realizable.

Next, let us consider a simple example of lop-frame.
Given a subset $S\subseteq \LopEff$ externally, we can naturally define a lop-frame associated with it.

\begin{definition}\label{def:stuni}
    For a subset $S\subseteq\LopEff$ and a local operator $j\in\LopEff$, 
    \begin{itemize}
        \item we define the following lop-frame $\stlop{S}$: 
        \[\stlop{S}(k) \mycoloneqq \bigcup_{\ell \in S} (\ell =_{\LopE} k).\]
        We call $\stlop{S}$ the \emph{standard lop-frame for $S$}.

        \item let $\upset{j}\mycoloneqq \{\, k\in \LopEff \mid j \leqEff k\,\}$ be the principal upper set of $j$ in $(\LopEff, \leqEff)$.
        We say that $S$ is \emph{uniform} if there exist realizers of $\isLopE[j]$ and $j\leqEff k$ for all local operators in $S$, uniformly.
        That is, the following conditions hold:
        \[\bigcap_{k\in S} \intinEff{\isLopE[k]}\neq \emptyset, \qquad \bigcap_{j\in S}\bigcap_{\, k\in S\cap \upset{j}} \intinEff{j\leq k}\neq \emptyset.\]

        \item $j$ is \emph{maximal in $S$} if $j\in S$ and $k \eqEff j$ hold for any $k\in S\cap \upset{j}$.
        
        \item $j$ is \emph{minimum in $S$} if $j\in S$ and $S\subseteq \upset{j}$ hold. \qedhere
    \end{itemize}
\end{definition}

The primary reason why we call $\stlop{S}$ standard is that 
the internal universal quantification over $\stlop{S}$ corresponds to the external universal quantification over $S$, as shown below.
The same applies to the property that $S$ is a subset of dense local operators.
\begin{lemma}\label{lem:stlop}
    Let $S$ be a subset of $\LopEff$ and let $j$ be a local operator.
    \begin{enumerate}
        \item Suppose that a predicate ${\Phi(k)}$ on $\OmegaEff^{\OmegaEff}$ is $\eqEff$-relational.
        Then, the following equivalence holds:
        \[\bigcap_{k\in \OmegaEff^{\OmegaEff}} (\stlop{S}(k)\land \intinEff{j\leq k} \rightarrow \Phi(k)) \eqEff \bigcap_{\ell \in S} (\intinEff{\isLopE[\ell]}\land \intinEff{j\leq \ell} \rightarrow \Phi[\ell]). \]

        \item $S\subseteq\DLopEff$ holds if and only if $\Eff\models (\stlop{S} \subDLopE)$ holds.
    \end{enumerate}
\end{lemma}
\begin{proof}
    We show (1).
    By the definition of the standard lop-frame $\stlop{S}$, the left-hand side is equivalent to:
    \[\bigcap_{k\in \OmegaEff^{\OmegaEff}} \bigcap_{\ell \in S} ((\ell \eqEff k) \rightarrow (\intinEff{\isLopE[k]}\land \intinEff{j\leq k} \rightarrow \Phi(k))).\]
    Note that $\intinEff{j\leq k}$, $\intinEff{\isLopE[k]}$, and $\Phi(k)$ are $\eqEff$-relational.
    Using the fact $((\bigcup_{x} \Psi(x)) \rightarrow p) = \bigcap_{x} (\Psi(x) \rightarrow p)$, the above predicate is also equivalent to:
    \[\bigcap_{\ell \in S} ( (\bigcup_{k\in \OmegaEff^{\OmegaEff}} (\ell \eqEff k) ) \rightarrow (\intinEff{\isLopE[\ell]}\land \intinEff{j\leq \ell} \rightarrow \Phi(\ell)) ).\]
    Since $\bigcup_{k} (\ell \eqEff k)$ contains a code for the identity function on $\N$ which is independent of $\ell$,
    we obtain the right-hand side.

    For (2), if $\ell$ is a dense operator, the claim follows from the fact that $\intinEff{\ell\leq \lnot\lnot}$ contains a code for the identity function on $\N$.
\end{proof}

Moreover, assuming that $S$ is uniform, the equivalence in Lemma~\ref{lem:stlop} (1) can be simplified.
In addition, maximality and minimumness are also characterized. 
\begin{lemma}\label{lem:uniformstlop}
    Let $S$ be a uniform subset of $\LopEff$ and let $j$ be a local operator.
    \begin{enumerate}
        \item Suppose that a predicate ${\Phi(k)}$ on $\OmegaEff^{\OmegaEff}$ is $\eqEff$-relational.
        Then, the following equivalence holds:
        \[\bigcap_{k\in \OmegaEff^{\OmegaEff}} (\stlop{S}(k) \land \intinEff{j\leq k} \rightarrow \Phi[k]) \eqEff \bigcap_{\ell \in S\cap\upset{j}} \Phi[\ell]. \]

        \item $j$ is maximal in $S$ if and only if $\Eff\models \forall k\geq_{\stlop{S}} j. (j =_{P\Omega} k)$ holds.

        \item $j$ is minimum in $S$ if and only if $\Eff\models \forall k\in\stlop{S}. (j\leq k)$ holds.
    \end{enumerate}
\end{lemma}
\begin{proof}
    (1) follows from the uniformity of $S$ and Lemma~\ref{lem:stlop} (1).
    (2) and (3) are easily verified from the uniformity for $j\leq k$ and Lemma~\ref{lem:uniformstlop} (1).
\end{proof}

Consequently, for a uniform subset $S$, 
the clause for implication in Definition~\ref{def:Preal} admits a simpler form.
Also, it follows from Lemma~\ref{lem:uniformstlop} (2) and Proposition~\ref{prop:maximal} that $\stlop{S}$-realizability at a maximal element $j$ in $S$ coincides with $j$-realizability.
\begin{corollary}\label{cor:stlopreal}
    Suppose that $S\subseteq\LopEff$ is uniform.
    Then, for the standard lop-frame $\stlop{S}$, the following equivalence holds: for any $j\in\LopEff$ and any $\Arithl$-formulas $\varphi$ and $\psi$
    \[\intinEff{j\Vdash_{\stlop{S}} (\varphi\rightarrow\psi)} \eqEff \bigcap_{k \in S\cap\upset{j}} (\intinEff{k\Vdash_{\stlop{S}} \varphi} \rightarrow k(\intinEff{k\Vdash_{\stlop{S}} \psi}) ).\]
    Furthermore, if $j$ is maximal in $S$, 
    the equivalence $\intinEff{j\Vdash_{\stlop{S}} \varphi} \eqEff \jreal{\varphi}$ holds.
\end{corollary}

\subsection{De Jongh-Goodman realizability}\label{subsec:deJonghGoodman}

As mentioned in the Introduction, the idea of combining realizability with forcing dates back to de Jongh and Goodman.
Their realizability is defined using a partial function on $\N$ as an oracle and a set of partial functions on $\N$ as a forcing poset \cite{Goodman78}.
To compare this with our $\PosetP$-realizability, let us focus on the formulation given in \cite[Section~4.6.2]{vOostenbook}.

\begin{definition}\label{def:deJonghGoodman}
    Let $\PFuncset$ denote the set of all partial functions on $\N$, 
    and let $\subseteq$ denote the extension relation on $\PFuncset$.
    Let the symbol $\rKlapp{f}$ denote a partial computable application using oracle $f\in\PFuncset$.

    For a subset $T\subseteq \PFuncset$, a partial function $f\in T$, a natural number $e\in \N$, and an $\Arithl$-sentence $\varphi$,
    we define the relation $\Tfrealat{f}{e} \varphi$ inductively as follows: 
    \[\begin{aligned}
        \Tfrealat{f}{e} (\varphi\rightarrow\psi) &\defarrow \forall g\in T \forall n\in\N. \\
        &(f \subseteq g \text{ and } \Tfrealat{g}{n}\varphi \implies e \rKlapp{g} n \downarrow \text{ and } \Tfrealat{g}{(e \rKlapp{g} n)}\psi), \\
        \Tfrealat{f}{e} (\forall x. \varphi[x]) &\defarrow \forall g\in T \forall n\in\N. \\
        &(f \subseteq g \implies e \rKlapp{g} n \downarrow \text{ and } \Tfrealat{g}{(e \rKlapp{g} n)} \varphi[n] ). 
    \end{aligned}\]
    The other clauses are defined similarly to Kleene realizability relative to $f$. 
\end{definition}

It should be noted that this definition is based on Kripke forcing and differs from the definition introduced by Goodman \cite{Goodman78}.
However, as pointed out in \cite[Section~4.6.2]{vOostenbook}, the original definition can be obtained by combining this with the double negation operator.
Hence the difference is not essential.
In fact, we can connect the above realizability with $\PosetP$-realizability via a uniform construction of local operators that is well understood in categorical realizability and topos theory.

It is well known that in an elementary topos $\CatE$, for any subobject $D\rightarrowtail X$, 
there exists the least local operator $j_{D}$ which forces $D$ dense.
This theorem was first shown by Joyal \cite[Theorem~3.57]{JhonstoneTopostheory}.
While various constructions for such a least local operator are known, we recall the method found in \cite[Lemma~5.4]{PittsPhD} and \cite[Proposition~16.3]{Hyland82}.
Here, we restrict our attention to the subobject defined by the graph of a partial function $f$.

\begin{definition}\label{def:monoandidem}
    Given a partial function $f\in\PFuncset$, 
    we define a predicate $D_{f}(n)\mycoloneqq \{\, \langle n, f(n) \rangle \,\}$ on $\N$; 
    if $f(n)$ is undefined, let $D_{f}(n)\mycoloneqq \emptyset$.
    Then the predicates $m_{f}$ and $j_{f}$ are defined as follows:
    \[\begin{aligned}
        m_{f}(p) &\mycoloneqq \bigcup_{n\in\N} (D_{f}(n) \rightarrow p), \\
        j_{f}(p) &\mycoloneqq \bigcap_{q\subseteq\N} ((p\rightarrow q)\land (m_{f}(q)\rightarrow q) \rightarrow q). 
    \end{aligned}\]
\end{definition}

Our first observation is that a realizer witnessing that $j_{f}$ is a local operator can be obtained uniformly in $f$.

\begin{observation}\label{obs:ImonandLop}
    For the predicates $m_{f}$ and $j_{f}$ in Definition~\ref{def:monoandidem}, 
    \begin{enumerate}
        \item $\bigcap_{f\in\PFuncset} \bigcap_{p, q\subseteq\N} ((p\rightarrow q) \rightarrow (m_{f}(p) \rightarrow m_{f}(q))) \neq \emptyset$.
        \item $\bigcap_{f\in\PFuncset} \intinEff{\isLopE[j_{f}]} \neq \emptyset \qedhere$. 
    \end{enumerate}
\end{observation}
\begin{proof}
    We verify the claim (1).
    In the internal language, $m_{f}$ can be described as 
    $m_{f}(p)\mycoloneqq \exists x:N. (x\in D_{f} \rightarrow p)$. 
    The fact that $m_{f}$ is internally monotone is provable entirely in the internal logic.
    Hence, there is no need to refer to oracle $f$.
    The claim (2) follows from the claim (1). 
\end{proof}

This uniform construction of local operator provides a bridge to degree theory.
We say that $f$ is \emph{$\mathrm{LT}$-reducible} to $g$, denoted by $f \leqreduce g$, if there exists $e\in \N$ such that $e\rKlapp{g}n \simeq f(n)$ holds for all $n\in \N$, 
where $\simeq$ denotes the Kleene equality.
(We adopt this terminology from \cite[Definition~2.14]{Kihara23}. 
Although this definition differs from the original one, the equivalence is easily verified.)
When $f$ and $g$ are characteristic functions of subsets of $\N$, 
the following equivalences are well known:
\begin{enumerate}
    \item $j_{f} \leqEff j_{g}$ holds if and only if $f \leqreduce g$ holds (e.g., \cite{Hyland82}).
    \item $\varphi$ is $j_{f}$-realizable if and only if $\varphi$ is Kleene realizable relative to $f$ (e.g., \cite{Phoa89}).
\end{enumerate}

Our second observation is that these two equivalences hold not only for characteristic functions but also for partial functions, 
and they hold effectively.
\begin{observation}\label{obs:relative}
    For the construction $j_{f}$ in Definition~\ref{def:monoandidem},
    the following hold:
    \begin{enumerate}
        \item $\bigcap_{f, g\in\PFuncset} (\intinEff{j_{f}\leq j_{g}} \leftrightarrow (f \leqreduce g)) \neq \emptyset$.
        \item $\bigcap_{f\in\PFuncset} \bigcap_{p, q\subseteq\N} ((p\rightarrow j_{f}(q)) \leftrightarrow (p\rightarrow^{f} q)) \neq \emptyset$.
    \end{enumerate}
    Here, $(f \leqreduce g)$ denotes the set of natural numbers witnessing $f \leqreduce g$, and 
    $(p\rightarrow^{f} q) \mycoloneqq \{\, e \mid \forall n\in p. (e\rKlapp{f}n\downarrow \text{ and } e\rKlapp{f}n\in q)\,\}$.
\end{observation}
The claim (1) is obtained by extending the proof of \cite[Theorem~17.2]{Hyland82} to partial functions, 
and (2) by extending the proof of \cite[Lemma in page.~2]{Phoa89}.
It is not difficult to verify that these proofs are effective and independent of oracle.
Strictly speaking, 
the original proofs were based not on $j_{f}$ but on the following local operator $j'_{f}$:
\[j'_{f}(p) \mycoloneqq \bigcap \{\, q\subseteq \N \mid p\land \{\,\ast\,\} \subseteq q \text{ and } m_{f}(q) \subseteq q \,\},\]
where $\ast\in \N$ is a code of the empty partial function.
However, in light of \cite[Proposition~5.6]{PittsPhD} and our Observation~\ref{obs:ImonandLop} (1), 
it is easy to verify that $j_{f}$ and $j'_{f}$ are equivalent uniformly in $f$.

Recently, Kihara provided an excellent framework for understanding the relationship between local operator and oracle \cite{Kihara23, Kihara24rethinking}.
He introduced the notion of \emph{bilayered function}, which is a significant generalization of oracle to describe all local operators in $\Eff$.
This framework enables us to analyze the order structure of local operators in $\Eff$ based on an imperfect information game.
In fact, the proof of (1) is essentially contained in \cite[Proposition~2.22 and Theorem~3.1]{Kihara23}.

Consequently, based on the observations above, 
we see that de Jongh-Goodman realizability is a special case of $\PosetP$-realizability under a mild assumption.
\begin{theorem}\label{thm:specialcase}
    Let $T\subseteq\PFuncset$ be a subset satisfying the following assumption $(\mathrm{A})$:
    \begin{equation*}
        \forall f, g\in T. (f\subseteq g \iff f \leqreduce g) \qquad (\mathrm{A}).
    \end{equation*}
    Then $S_{T}\mycoloneqq \{\, j_{f} \mid f\in T\,\}$ is a uniform subset of $\LopEff$.

    Thus, for the standard lop-frame $\PosetP_{T} \mycoloneqq \PosetP_{S_{T}}$, the following equivalence holds:
    for any partial function $f\in T$ and any $\Arithl$-sentence $\varphi$, 
    \[\intinEff{j_{f}\Vdash_{\PosetP_{T}} \varphi} \eqEff \{\, e\mid \Tfrealat{f}{e}\varphi \,\}.\]
\end{theorem}
\begin{proof}
    Clearly, there exists a natural number $r$ such that for any $f$, $g\in\PFuncset$, $r$ is a code witnessing $f\leqreduce g$ whenever $f \subseteq g$.
    That is, 
    \[\bigcap_{f \in\PFuncset}\bigcap_{g \in\PFuncset \text{ s.t. } f\subseteq g} (f\leqreduce g) \neq \emptyset.\]
    Thus, if $T$ satisfies the assumption $(A)$, the uniformity of $S_{T}$ follows from Observation~\ref{obs:ImonandLop} (2) and Observation~\ref{obs:relative} (1). 

    The latter equivalence can be shown by Corollary~\ref{cor:stlopreal} and Observation~\ref{obs:relative} (2).
\end{proof}

\section{Semi-classical axioms}\label{sec:semiclassical}

In this section, we investigate our translation for specific classes of formulas.
In particular, with a view towards future proof-theoretic applications, we focus on \emph{semi-classical axioms}, which are frequently studied in constructive mathematics.
Our main approach is to identify sufficient conditions under which our translation of a given axiom becomes equivalent to its usual $j$-translation.
This approach allows us to use existing results on the $j$-translation.

To give a systematic investigation, in Section~\ref{subsec:equivalence}, \ref{subsec:trpcldense}, and \ref{subsec:semiclassicalHA}, 
we return to the general setting of an arbitrary elementary topos.
Hence, all statements presented there are verified in the internal logic.
Finally, in Section~\ref{subsec:separation}, we apply these general results to $\PosetP$-realizability.

\subsection{Equivalence condition}\label{subsec:equivalence}
We begin by discussing the conditions under which our translation $j\forceinP\varphi$ is equivalent to the internal $j$-translation $\intjtrans{\varphi}[j]$.
In general, this equivalence depends on both the syntactic complexity of $\varphi$ and the structure of the lop-frame $\PosetP$.
In the following, we assume $j\in\PosetP$ for simplicity.
For convenience, we introduce a predicate $\EquivP{\varphi}$ that expresses the equivalence of these two translations.
\begin{definition}\label{def:EquivP} 
    For an $\langX$-formula $\varphi[\vec{x}]$, we define the formula $\EquivP{\varphi}$ with a free variable $\PosetP:P(P\Omega)$ as follows: 
    \[\EquivP{\varphi} \mycoloneqq \forall j\inPosetP \forall \vec{x}:X. (j\forceinP\varphi[\vec{x}] \leftrightarrow \intjtrans{\varphi}[j, \vec{x}]). \qedhere\]
\end{definition}

It is easy to observe that implication-free formulas always satisfy this equivalence.
\begin{proposition}\label{prop:impfreeE}
    Let $\varphi[\vec{x}]$ be an implication-free $\langX$-formula (that is, a formula containing neither implication $\rightarrow$ nor negation $\lnot$).
    Then, 
    \[\CatE\models \forall \PosetP\subLopE. \EquivP{\varphi}.\] 
\end{proposition}
\begin{proof}
    The base case and the induction steps for $\land$, $\lor$, and $\exists$ follow directly from the definitions of $j\forceinP\varphi$ and $\intjtrans{\varphi}[j]$.
    The step for the universal quantification $\forall$ follows from the assumption $j\inPosetP$ and Corollary~\ref{cor:jinPconstantdomain}.
\end{proof}

However, non-constructive axioms such as the law of excluded middle $\forall \vec{x}:X. \varphi\lor\lnot\varphi$ and the double negation elimination $\forall \vec{x}:X. \lnot\lnot\varphi\rightarrow\varphi$ are not implication-free.
Therefore, the equivalence does not hold for them in general.
To find the sufficient conditions for these equivalences to hold, 
we now introduce the following predicates for the internal $j$-translation.

\begin{definition}\label{def:MNneg}
    Let $\varphi[\vec{x}]$ be an $\langX$-formula.
    \[\begin{aligned}
        \MonoP{\varphi}\mycoloneqq \forall j\inPosetP \forall \vec{x}:X. \forall k\geqinP{j}. (\intjtrans{(\varphi)}[j, \vec{x}] \rightarrow \intjtrans{(\varphi)}[k, \vec{x}]) \\
        \NonoP{\varphi}\mycoloneqq \forall j\inPosetP \forall \vec{x}:X. \forall k\geqinP{j}. (\intjtrans{(\varphi)}[k, \vec{x}] \rightarrow \intjtrans{(\varphi)}[j, \vec{x}]) \\
    \end{aligned}\]
    We also define $\intlnot{j}\varphi\mycoloneqq \varphi\rightarrow j\bot$.
\end{definition}

Note that $\intjtrans{(\lnot\varphi)}[j]\myeqq \intlnot{j}(\intjtrans{\varphi}[j])$ and $j\forceinP(\lnot\varphi) \myeqq \forall k\geqinP{j}.\intlnot{k}(k\forceinP\varphi)$ hold.
The following lemma summarizes the properties for negated formulas $\lnot\varphi$ and double negated formulas $\lnot\lnot\varphi$.
\begin{lemma}\label{lem:EMN}
    For any $\langX$-formula $\varphi[\vec{x}]$, the following implications hold: 
    \begin{enumerate}
        \item $\CatE\models \forall\PosetP\subLopE. \NonoP{\lnot\lnot\varphi} \rightarrow \MonoP{\lnot\varphi}$; 
        \item $\CatE\models \forall\PosetP\subLopE. (\EquivP{\varphi}\land \MonoP{\lnot\varphi}) \rightarrow \EquivP{\lnot\varphi}$;
        \item $\CatE\models \forall\PosetP\subLopE. (\EquivP{\varphi}\land \MonoP{\lnot\varphi}\land \MonoP{\lnot\lnot\varphi}) \rightarrow \EquivP{\lnot\lnot\varphi}$; 
        \item $\CatE\models \forall\PosetP\subLopE. (\EquivP{\varphi}\land \NonoP{\lnot\lnot\varphi}) \rightarrow \MonoP{\lnot\lnot\varphi\rightarrow\varphi}$.
    \end{enumerate}
\end{lemma} 
\begin{proof}
    In the following, we fix $j$, $k\in\PosetP$ with $j\leq k$ and omit the free variables $\vec{x}$. 

    For (1), we assume $\NonoP{\lnot\lnot\varphi}$ and $\intjtrans{(\lnot\varphi)}[j]$.
    Note that $\forall p:\Omega. \intlnot{j}\intlnot{j}\intlnot{j}p \leftrightarrow \intlnot{j}p$ and $\forall p, q:\Omega. (p\rightarrow q) \rightarrow (\intlnot{j}q \rightarrow \intlnot{j}p)$ hold.
    Then, $\intlnot{j}\intjtrans{(\lnot\lnot\varphi)}[k]$ follows from $\NonoP{\lnot\lnot\varphi}$. 
    Since $\forall p:\Omega. \intlnot{j}p\rightarrow \intlnot{k}p$, we obtain $\intlnot{k}\intjtrans{(\lnot\lnot\varphi)}[k]$, which is equivalent to $\intjtrans{(\lnot\varphi)}[k]$.
    This proves that $\NonoP{\lnot\lnot\varphi}$ implies $\intjtrans{(\lnot\varphi)}[j] \rightarrow \intjtrans{(\lnot\varphi)}[k]$.

    For (2), we assume $\EquivP{\varphi}$ and $\MonoP{\lnot\varphi}$.
    By $\MonoP{\lnot\varphi}$, the formula $\intjtrans{(\lnot\varphi)}[j]$ is equivalent to $\forall k\geqinP{j}. \intlnot{k} (\intjtrans{\varphi}[k])$.
    It then follows from $\EquivP{\varphi}$ that the latter is equivalent to $\forall k\geqinP{j}. \intlnot{k}(k\forceinP \varphi)$, 
    which is also equivalent to $j\forceinP\varphi$. 

    For (3), we assume $\EquivP{\varphi}$, $\MonoP{\lnot\varphi}$, and $\MonoP{\lnot\lnot\varphi}$.
    By (2), the formula $j\forceinP \lnot\lnot\varphi$ is equivalent to $\forall k\geqinP{j}. \intlnot{k} \intjtrans{(\lnot\varphi)}[k]$.
    Using the assumption $\MonoP{\lnot\lnot\varphi}$, this becomes equivalent to $\intjtrans{(\lnot\lnot\varphi)}[j]$.

    To show (4), we assume $\EquivP{\varphi}$, $\NonoP{\lnot\lnot\varphi}$, and $\intjtrans{(\lnot\lnot\varphi\rightarrow\varphi)}[j]$.
    First, $\NonoP{\lnot\lnot\varphi}$ and $\intjtrans{(\lnot\lnot\varphi\rightarrow\varphi)}[j]$ imply the following implications: 
    \[\intjtrans{(\lnot\lnot\varphi)}[k] \rightarrow \intjtrans{(\lnot\lnot\varphi)}[j] \rightarrow \intjtrans{(\varphi)}[j]\]
    Using $\EquivP{\varphi}$ and Lemma~\ref{lem:monotonicity},
    $\intjtrans{(\varphi)}[k]$ follows from $\intjtrans{(\varphi)}[j]$. 
    This shows the implication $\intjtrans{(\lnot\lnot\varphi)}[k] \rightarrow \intjtrans{\varphi}[k]$, which is precisely $\intjtrans{(\lnot\lnot\varphi\rightarrow\varphi)}[k]$.
    Therefore, $\MonoP{\lnot\lnot\varphi\rightarrow\varphi}$ follows from $\EquivP{\varphi}$ and $\NonoP{\lnot\lnot\varphi}$.
\end{proof}

These lemmas provide sufficient conditions for $\EquivP{\forall \vec{x}:X. (\varphi\lor\lnot\varphi)}$ and $\EquivP{\forall \vec{x}:X. (\lnot\lnot\varphi\rightarrow\varphi)}$.
\begin{proposition}\label{prop:MNdneg}
    Let $\varphi[\vec{x}]$ be an $\langX$-formula.
    The following implications hold in $\CatE$:
    \begin{enumerate}
        \item $\CatE\models \forall\PosetP\subLopE. (\EquivP{\varphi}\land \MonoP{\lnot\varphi}) \rightarrow \EquivP{\forall \vec{x}:X. (\varphi\lor\lnot\varphi)}$; 
        \item $\CatE\models \forall\PosetP\subLopE. (\EquivP{\varphi}\land \MonoP{\lnot\lnot\varphi}\land \NonoP{\lnot\lnot\varphi}) \rightarrow \EquivP{\forall \vec{x}:X. (\lnot\lnot\varphi\rightarrow\varphi)}$.
    \end{enumerate}
\end{proposition}
\begin{proof}
    The claim (1) follows immediately from Corollary~\ref{cor:jinPconstantdomain} and Lemma~\ref{lem:EMN} (2).
    
    We show (2).
    By Corollary~\ref{cor:jinPconstantdomain}, $j\forceinP \forall \vec{x}:X. (\lnot\lnot\varphi\rightarrow\varphi)$ is equivalent to:
    \[\forall \vec{x}:X \forall k\geqinP{j}. (k\forceinP\lnot\lnot\varphi [\vec{x}] \rightarrow k\forceinP \varphi[\vec{x}] ).\]
    We now apply Lemma~\ref{lem:EMN} (1) and (3). 
    Under the assumptions $\EquivP{\varphi}$, $\MonoP{\lnot\lnot\varphi}$ and $\NonoP{\lnot\lnot\varphi}$, 
    the above formula is equivalent to:
    \[\forall \vec{x}:X \forall k\geqinP{j}. (\intjtrans{(\lnot\lnot\varphi)}[k, \vec{x}] \rightarrow \intjtrans{\varphi}[k, \vec{x}]).\]
    By Lemma~\ref{lem:EMN} (4), the last formula is equivalent to $\intjtrans{(\forall \vec{x}:X. (\lnot\lnot\varphi\rightarrow\varphi))}[j]$.
\end{proof}

\subsection{Transparency, closedness and denseness}\label{subsec:trpcldense}

Lemma~\ref{lem:EMN} (1) ensures that the condition $\EquivP{\varphi}\land \MonoP{\lnot\lnot\varphi}\land \NonoP{\lnot\lnot\varphi}$ suffices to establish both 
$\EquivP{\forall \vec{x}:X. (\lnot\lnot\varphi\rightarrow\varphi)}$ and $\EquivP{\forall \vec{x}:X. (\varphi\lor\lnot\varphi)}$.
However, it remains unclear which formulas satisfy $\MonoP{\lnot\lnot\varphi}\land \NonoP{\lnot\lnot\varphi}$.
Here, we reveal that \emph{transparency}, a notion introduced in the author's previous study \cite{Nakata24},
provides a sufficient condition for this property.

The notion of transparency was originally found in the investigation of the least local operator for a given formula \cite{Nakata24}.
In that study, $j$-transparency with respect to a local operator $j$ in a topos $\CatE$ was defined via the associated sheaf functor $\jshfunc$ as follows: 
\begin{equation*}
    \jsh{\subinp{\varphi}} = \subinp{\shjtrans{\varphi}} \text{ in } \SubEj(\jsh{X}) \quad\quad [j\text{-transparency}].
\end{equation*}
In other words, a formula $\varphi$ is $j$-transparent if its $\jshfunc$-translation $\shjtrans{\varphi}$ (Definition~\ref{def:shjtrans}) 
corresponds to the subobject obtained by a single application of $\jshfunc$ to $\subinp{\varphi}$.

Recalling the correspondence observed in Section~\ref{subsec:shjtrans}, 
the pullback functor along the unit $\eta_{X}\colon X\to \jsh{X}$ shows that 
the equation above is equivalent to $j\subinp{\varphi} = \subinp{\extjtrans{\varphi}}$ in $\CljSubE(X)$.
Thus, $j$-transparency can be expressed in the internal language as the following equivalence: 
\[\forall \vec{x}:X. (j\varphi[\vec{x}] \leftrightarrow \extjtrans{\varphi}[\vec{x}]).\]
This property may be seen as an analogue of Glivenko's theorem for the double negation translation 
(which states that for any propositional formula $\varphi$, the translation $\varphi^{\lnot\lnot}$ is equivalent to $\lnot\lnot\varphi$ in intuitionistic propositional logic; see \cite{ConstructivismvolI}).
While \cite{Nakata24} utilized this property to discuss the existence of the least local operator for some semi-classical axioms, 
Proposition~\ref{prop:trpimpMN} below shows that transparency, combined with certain conditions, also implies $\MonoP{\lnot\lnot\varphi}\land \NonoP{\lnot\lnot\varphi}$.

For the sake of generality, let us introduce a formula $\Trp{\varphi}{j}{k}$ expressing that $\varphi$ is $k$-transparent relative to $j$.
We also consider a corresponding closedness condition $\Cl{\varphi}{j}{k}$. 

\begin{definition}\label{def:trpandcl}
    For an $\langX$-formula $\varphi[\vec{x}]$, we define 
    \[\begin{aligned}
        \Trp{\varphi}{j}{k} &\mycoloneqq \forall \vec{x}:X. (k\intjtrans{\varphi}[j, \vec{x}] \leftrightarrow \intjtrans{\varphi}[k, \vec{x}]), \\
        \Cl{\varphi}{j}{k} &\mycoloneqq \forall \vec{x}:X. (\intjtrans{\varphi}[j, \vec{x}] \leftrightarrow k\intjtrans{\varphi}[j, \vec{x}]). 
    \end{aligned}\]
\end{definition}

We first verify that the following closure properties, shown in \cite[Lemma~25]{Nakata24}, also hold in the internal logic.
\begin{lemma}\label{lem:trp}
    Let $R$ be an atomic $\langX$-formula, and let $\varphi$ and $\psi$ be $\langX$-formulas.
    Then the following implications hold:
    \begin{enumerate}
        \item $\CatE\models \forall j, k\inLopE. ((j\leq k) \rightarrow \Trp{R}{j}{k})$.
        \item $\CatE\models \forall j, k\inLopE. (\Trp{\varphi}{j}{k} \land \Trp{\psi}{j}{k} \rightarrow \Trp{\varphi\land\psi}{j}{k})$.
        \item $\CatE\models \forall j, k\inLopE. (\Trp{\varphi}{j}{k} \land \Trp{\psi}{j}{k} \land (j\leq k) \rightarrow (\Trp{\varphi\lor\psi}{j}{k} \land \Trp{\exists x:X. \varphi}{j}{k}))$.
        \item $\CatE\models \forall j, k\inLopE. (\Trp{\varphi}{j}{k} \land \Trp{\psi}{j}{k} \land \Cl{\psi}{j}{k} \rightarrow (\Trp{\varphi\rightarrow\psi}{j}{k} \land \Trp{\forall x:X. \varphi}{j}{k}))$.
    \end{enumerate}
\end{lemma}
\begin{proof}
    Since the other cases can be shown similarly, 
    we only see the case of $\lor$ in (3) and the case of $\rightarrow$ in (4).
    In the following, we assume that $\langX$-formulas do not contain free variables.

    For the case of $\lor$ in (3), assume that $\Trp{\varphi}{j}{k}$, $\Trp{\psi}{j}{k}$, and $j\leq k$.
    Then we have the following equivalences for $k\intjtrans{(\varphi\lor\psi)}[j]$: 
    \[\begin{aligned}
        &k\intjtrans{(\varphi\lor\psi)}[j] =  kj (\intjtrans{\varphi}[j] \lor \intjtrans{\psi}[j]) \\
        \leftrightarrow &k(\intjtrans{\varphi}[j] \lor \intjtrans{\varphi}[j]) \leftrightarrow k(k\intjtrans{\varphi}[j] \lor k\intjtrans{\varphi}[j]).
    \end{aligned}\]
    Since $\Trp{\varphi}{j}{k}$ and $\Trp{\psi}{j}{k}$ hold, the last formula is equivalent to $\intjtrans{(\varphi\lor\psi)}[k]$.

    Next, consider the case of $\rightarrow$ in (4).
    Suppose that $\Trp{\varphi}{j}{k}$, $\Trp{\psi}{j}{k}$, and $\Cl{\psi}{j}{k}$.
    By $\Cl{\psi}{j}{k}$, we obtain the following equivalence for $k\intjtrans{(\varphi\rightarrow\psi)}[j]$: 
    \[\begin{aligned}
        k\intjtrans{(\varphi\rightarrow\psi)}[j] &=  k (\intjtrans{\varphi}[j] \rightarrow \intjtrans{\psi}[j]) \\
        &\leftrightarrow ( k\intjtrans{\varphi}[j] \rightarrow k\intjtrans{\psi}[j]).
    \end{aligned}\]
    Since $\Trp{\varphi}{j}{k}$ and $\Trp{\psi}{j}{k}$ hold, the last formula is equivalent to $\intjtrans{(\varphi\rightarrow\psi)}[k]$.
\end{proof}

Lemma~\ref{lem:trp} ensures that under the assumption $(j\leq k)$, all \emph{coherent} formulas (that is, formulas containing only $\land$, $\lor$, and $\exists$) 
are $k$-transparent relative to $j$.
Furthermore, as observed in \cite{Nakata24} and essentially in \cite{Caramello14}, 
\emph{denseness} allows us to extend transparency to a broader class of formulas.

\begin{definition}\label{def:dense}
    Let $\lnot\lnot$ denote the double negation operator.
    We define a closed term $\DLopE \mycoloneqq \{\, j:P\Omega \mid \isLopE[j] \land j\leq \lnot\lnot \,\}$.
\end{definition}

\begin{remark}\label{rem:dense}
    As is well known, the formula $j\leq \lnot\lnot$ is internally equivalent to $j(\bot)\leftrightarrow \bot$ for any $j\inLopE$.
    Therefore, $\forall j\inDLopE \forall p:\Omega. (\intlnot{j}p \leftrightarrow \lnot p)$ holds in $\CatE$.
\end{remark}

In light of Remark~\ref{rem:lopframe} and Remark~\ref{rem:dense}, 
we immediately obtain the following lemma.
\begin{lemma}\label{lem:denseandDNE}
    Let $R$ be an atomic $\langX$-formula, and let $\varphi$ be an $\langX$-formula.
    The following implications hold: 
    \begin{enumerate}
        \item $\CatE\models \forall j \inDLopE. ((\DNEAx{R}) \rightarrow \intjtrans{(\DNEAx{R})}[j])$.
        \item $\CatE\models \forall j, k\inDLopE. (\intjtrans{(\DNEAx{\varphi})}[j] \rightarrow \Cl{\varphi}{j}{k})$.
    \end{enumerate}
\end{lemma} 

Consequently, transparency and denseness can guarantee the condition $\MonoP{\lnot\lnot\varphi}\land \NonoP{\lnot\lnot\varphi}$ as follows: 
\begin{proposition}\label{prop:trpimpMN}
    For any $\langX$-formula $\varphi[\vec{x}]$, 
    \[\CatE\models \forall \PosetP\subDLopE \forall j\inLopE. ((\forall k\inPosetP. \Trp{\varphi}{j}{k}) \rightarrow \MonoP{\lnot\lnot\varphi}\land \NonoP{\lnot\lnot\varphi}).\]
\end{proposition}
\begin{proof}
    Suppose that $\PosetP\subDLopE$ satisfies the antecedent.
    For $k\inPosetP$, we have the following equivalences: 
    \[\begin{aligned}
        &\intjtrans{(\lnot\lnot\varphi)}[k] \leftrightarrow \lnot\lnot\intjtrans{\varphi}[k] \\
        \leftrightarrow &\lnot\lnot k\intjtrans{\varphi}[j] \leftrightarrow \lnot\lnot \intjtrans{\varphi}[j],  
    \end{aligned}\]
    where the first and third equivalences follow from the denseness of $k$, 
    and the second follows from $\Trp{\varphi}{j}{k}$.
    Therefore, for any $k$, $\ell\inPosetP$, the equivalence $\intjtrans{(\lnot\lnot\varphi)}[k] \leftrightarrow \lnot\lnot \intjtrans{\varphi}[j] \leftrightarrow \intjtrans{(\lnot\lnot\varphi)}[\ell]$ holds.
    This proves $\MonoP{\lnot\lnot\varphi}\land \NonoP{\lnot\lnot\varphi}$.
\end{proof}

\subsection{Semiclassical axioms over $\HA$}\label{subsec:semiclassicalHA}

We now investigate the equivalence $\EquivP{\empty}$ for semi-classical axioms over Heyting arithmetic $\HA$.
Typical examples are the restricted versions of the double negation elimination ($\DNE$) and the law of excluded middle ($\LEM$), which were systematically investigated in \cite{ABHK04, FujiwaraKurahashi22}.
We begin by recalling the standard notation.
\begin{definition}\label{def:Arithhier}
    We inductively define the classes $\Sigma_n$, $\Pi_n$ of $\Arithl$-formulas as follows:
    $\Sigma_0 = \Pi_0$ is the set of all quantifier-free formulas,
    while $\Sigma_{n+1}$ and $\Pi_{n+1}$ are defined by 
    $\Sigma_{n+1} \coloneqq
    \{\, \exists x_1 \cdots \exists x_k \varphi \mid \varphi\in \Pi_{n}, \, 0\leq k\,\}$ and 
    $\Pi_{n+1} \coloneqq
    \{\, \forall x_1 \cdots \forall x_k \varphi \mid \varphi\in \Sigma_{n}, \, 0\leq k\,\}$.
    The class $\Pi_n \lor \Pi_n$ denotes
    the set of formulas of the form $\varphi\lor \psi$
    where $\varphi$, $\psi\in \Pi_n$.
    For a class $\Gamma$ of $\Arithl$-formulas, 
    the following axiom schemes restricted to $\Gamma$ is defined: 
    \[\begin{aligned}
        \AxScheme{\Gamma}{DNE} &\mycoloneqq \{\, \forall \vec{x}. (\lnot\lnot\varphi \rightarrow \varphi) \mid \varphi[\vec{x}] \in \Gamma \,\}, \\
        \AxScheme{\Gamma}{LEM} &\mycoloneqq \{\, \forall \vec{x}. (\varphi \lor \lnot\varphi) \mid \varphi[\vec{x}] \in \Gamma \,\}. 
    \end{aligned}\]
\end{definition}

\begin{remark}
    By formalizing Post's theorem in $\HA$, 
    one obtains a \emph{universal formula} $\univfml{\Sigma_{n}}[e, x]$ for $\Sigma_{n}$ for each $n\geq 1$.
    That is, for any $\Sigma_{n}$ formula $\psi[x]$, there exists a numeral $e_{\psi}$ such that $\HA \vdash \forall x. (\univfml{\Sigma_{n}}[e_{\psi}, x] \leftrightarrow \psi[x])$.
    Similar results hold for $\Pi_{n}$ and $\Pi_{n}\lor \Pi_{n}$.
    For example, we may define $\univfml{\Sigma_{1}} \mycoloneqq \exists w. T[e, x, w]$ and $\univfml{\Pi_{1}}\mycoloneqq \forall w. \lnot T[e, x, w]$, where $T$ is Kleene's $T$-predicate.

    Thus, if we assume $\AxScheme{\Sigma_{0}}{DNE} \mycoloneqq (0=0)$, 
    then for any natural number $n$, the axiom scheme $\AxScheme{\Sigma_{n}}{DNE}$ can be regarded as a single formula in $\HA$.
    Accordingly, we will use $\intjtrans{(\AxScheme{\Sigma_{n}}{DNE})}[j]$ as a single formula.
    Similarly for $\Trp{\Sigma_{n+1}}{j}{k}$ and $\Trp{\Pi_{n+1}}{j}{k}$.
\end{remark}

Since the classes $\Sigma_{n}$ and $\Pi_{n}$ in Definition~\ref{def:Arithhier} consist of prenex normal formulas,  
Proposition~\ref{prop:impfreeE} immediately implies the following corollary.

\begin{corollary}\label{cor:PNFisE}
    Let $\CatE$ be an elementary topos with a natural numbers object.
    For any $\Arithl$-formula $\varphi$ in $\Sigma_{n}$, $\Pi_{n}$, or $\Pi_{n}\lor\Pi_{n}$ for a natural number $n\in\N$, 
    the equivalence $\CatE\models \forall\PosetP\subLopE. \EquivP{\varphi}$ holds.
\end{corollary}

Furthermore, the following lemma can be proven internally.
This result is a $j$-relativized version of \cite[Lemma~29]{Nakata24}.
\begin{proposition}\label{prop:trpladder}
    Let $\CatE$ be an elementary topos with a natural numbers object.
    For any natural number $n$, the following implication holds:
    \[\CatE\models \forall j, k\inDLopE. ( \intjtrans{(\AxScheme{\Sigma_{n}}{DNE})}[j] \land (j\leq k) \rightarrow (\Trp{\Pi_{n+1}}{j}{k} \land \Trp{\Sigma_{n+1}}{j}{k}) ). \]
\end{proposition}
\begin{proof}
    By induction on $n$.
    First, we show the base case $n=0$.
    Let $R$ be an arbitrary atomic $\Arithl$-formula.
    Then the assumption $j\leq k$ and Lemma~\ref{lem:trp} (1) imply $\Trp{R}{j}{k}$.
    From Lemma~\ref{lem:trp} (3), we have $\Trp{\Sigma_{1}}{j}{k}$.
    Since $\DNEAx{R}$ is provable in $\HA$, it holds in $\CatE$.
    By Lemma~\ref{lem:denseandDNE} (1) and (2), we obtain $\Cl{R}{j}{k}$.
    Combined this with $\Trp{R}{j}{k}$ and Lemma~\ref{lem:trp} (4), we get $\Trp{\Pi_{1}}{j}{k}$.

    Next, consider the induction step.
    Since $\intjtrans{(\AxScheme{\Sigma_{n+1}}{DNE})}[j]$ clearly implies $\intjtrans{(\AxScheme{\Sigma_{n}}{DNE})}[j]$, 
    $\Trp{\Pi_{n+1}}{j}{k}$ and $\Trp{\Sigma_{n+1}}{j}{k}$ follow from the induction hypothesis.
    In addition, we obtain $\Cl{\Sigma_{n+1}}{j}{k}$ by Lemma~\ref{lem:denseandDNE} (2).
    The rest of the argument is the same as the base case $n=0$ by replacing $R$ with $\Sigma_{n+1}$.
\end{proof}

Our goal has been to characterize the conditions under which a lop-frame $\PosetP$ satisfies $\EquivP{\empty}$ for semi-classical axioms over $\HA$.
Combining the results above, we obtain the following sufficient condition.
\begin{theorem}\label{thm:sufconforSCA}
    Let $\CatE$ be an elementary topos with a natural numbers object and $n$ be a natural number.
    Suppose that $\Gamma$ is one of the classes $\{\, \Sigma_{n+1}, \Pi_{n+1}, \Pi_{n+1}\lor\Pi_{n+1}, \Sigma_{n+2}\,\}$ and 
    $\mathrm{Ax}$ denote the axiom scheme $\mathbf{DNE}$ or $\mathbf{LEM}$.
    Then the following implication holds: 
    \[\CatE\models \forall \PosetP\subDLopE \forall j\inPosetP. (\intjtrans{(\AxScheme{\Sigma_{n}}{DNE})}[j] \land \forall k\inPosetP. (j\leq k) \rightarrow \EquivP{\Gamma\text{-}\mathrm{Ax}}).\]
\end{theorem}
\begin{proof}
    Assume that $\PosetP$ and $j$ satisfy the antecedent.
    Let $\varphi[\vec{x}]$ be an $\Arithl$-formula in $\Gamma$.
    By Corollary~\ref{cor:PNFisE}, we have $\EquivP{\varphi}$.
    From the assumption $\forall k\inPosetP. (j\leq k)$, Proposition~\ref{prop:trpladder}, and Lemma~\ref{lem:trp} (3) and (4), 
    we obtain $\forall k\inPosetP. \Trp{\varphi}{j}{k}$.
    Consequently, both $\EquivP{\forall \vec{x}. (\varphi\lor\lnot\varphi)}$ and $\EquivP{\forall \vec{x}. (\lnot\lnot\varphi\rightarrow\varphi)}$ follow from Proposition~\ref{prop:trpimpMN} and Proposition~\ref{prop:MNdneg}.
\end{proof}

\subsection{Separation by $\PosetP$-realizability}\label{subsec:separation}

Using the results establishing so far, we now demonstrate a separation proof for semi-classical axioms via $\PosetP$-realizability.
We here emphasize the distinction between our approach and the usual $j$-realizability.
Regarding $j$-realizability, the following result is well known (see, for example, \cite{ABHK04}; a more detailed analysis appears in \cite[Corollary~47]{Nakata24}).
\begin{example}\label{ex:jrealforSCA}
    For any natural number $n$, let $\Tjump{n}$ denote the degree of $n$-th Turing jump, and let $j_{n}$ be the local operator corresponding to ${\Tjump{n}}$.
    Then $\AxScheme{\Sigma_{n+1}}{DNE}$ and $\AxScheme{\Pi_{n}\lor\Pi_{n}}{DNE}$ are $j_{n}$-realizable,
    while $\AxScheme{\Pi_{n+1}\lor\Pi_{n+1}}{DNE}$ is not.
\end{example}

Combining the soundness of $j$-realizability for $\HA$, 
this result implies that $\HA + \AxScheme{\Sigma_{n+1}}{DNE}$ does not prove $\AxScheme{\Pi_{n+1}\lor\Pi_{n+1}}{DNE}$.
In this way, $j$-realizabilities applies to separation problems.
Indeed, Kihara \cite{Kihara25degreesincomp} establishes numerous separation results for axioms over intuitionistic set theory $\IZF$ by combining $j$-realizability with McCarty realizability.

However, the following proposition indicates a limitation of $j$-realizability regarding the double negated variants of axioms.
\begin{proposition}\label{prop:limitation}
    For any $\Arithl$-sentence $\varphi$ and any local operator $j$ in $\Eff$,
    the following implication holds externally:
    \[\Eff \models \extjtrans{(\lnot\lnot \varphi)} \implies \Eff\models \extjtrans{\varphi}.\]
\end{proposition}
\begin{proof}
    The case where $j$ is degenerate is trivial.
    Suppose that $j$ is non-degenerate and $\Eff \models \extjtrans{(\lnot\lnot \varphi)}$.
    Since $j$ is dense in $\Eff$, 
    $\extjtrans{(\lnot\lnot \varphi)}$ is equivalent to $\lnot\lnot(\extjtrans{\varphi})$, which implies that the predicate $\lnot\lnot\intinEff{\extjtrans{\varphi}}$ is realizable.
    As we have seen in Remark~\ref{rem:dense}, $\intinEff{\extjtrans{\varphi}}$ is also realizable.
    We thus obtain $\Eff\models \extjtrans{\varphi}$.
\end{proof}

This limitation reflects the fact that if $j$ is non-degenerate, the lattice $\mathrm{Sub}_{\Eff_{j}}(1)$ of subterminal objects is boolean.
In constructive mathematics, the double negated variant $\lnot\lnot T \mycoloneqq \{\, \lnot\lnot\varphi \mid \varphi\in T\,\}$ of an axiom scheme $T$ often arises naturally.
For instance, \cite{FujiwaraKohlenbach18} provides a systematic study on such variants.
Despite their importance, Proposition~\ref{prop:limitation} shows that $j$-realizability is not effective for separating these variants.

In contrast, our $\PosetP$-realizability can separate such double negated variant, as shown below.
\begin{theorem}\label{thm:separation}
    For any natural number $n$,
    there exists a lop-frame $\PosetP_{n}$ and a local operator $j_{n}$ in $\Eff$ such that 
    $\AxScheme{\Sigma_{n+1}}{DNE} + \lnot\lnot(\AxScheme{\Pi_{n+1}\lor\Pi_{n+1}}{DNE})$ is $\PosetP_{n}$-realizable at $j_{n}$, 
    while $\AxScheme{\Pi_{n+1}\lor\Pi_{n+1}}{DNE}$ is not.
    Therefore, the following underivability result holds:
    \[\HA + \AxScheme{\Sigma_{n+1}}{DNE} + \lnot\lnot(\AxScheme{\Pi_{n+1}\lor\Pi_{n+1}}{DNE}) \not\vdash \AxScheme{\Pi_{n+1}\lor\Pi_{n+1}}{DNE}.\]
\end{theorem}

This theorem indicates not only that $\PosetP$-realizability is useful for separation problems but also that it is a genuine generalization of $j$-realizability. 
\begin{corollary}\label{cor:comparison}
    For any local operator $j$ in $\Eff$, 
    $\PosetP_{n}$-realizability at $j_{n}$ in Theorem~\ref{thm:separation} does not coincide with $j$-realizability.
\end{corollary}

Most lemmas required for the proof of Theorem~\ref{thm:separation} have already been established in the previous sections.
It remains only to note the following lemma.
\begin{lemma}\label{lem:Prealdbnegsnt}
    Let $\varphi$ be an $\langX$-sentence.
    Let $S\subseteq\DLopEff$ be a uniform subset of dense local operators, and 
    consider the standard lop-frame $\stlop{S}$ for $S$.
    Suppose that a local operator $j\in S$ satisfies the following condition: for any $k\in S \cap\upset{j}$, 
    there exists $\ell\in S \cap\upset{k}$ such that $\varphi$ is $\stlop{S}$-realizable at $\ell$.

    Then, $\lnot\lnot\varphi$ is $\PosetP$-realizable at $j$.
\end{lemma}
\begin{proof}
    Recall from Lemma~\ref{lem:stlop} (2) that $\Eff \models (\stlop{S} \subDLopE)$ holds.
    Then $j\Vdash_{\stlop{S}} \lnot\lnot\varphi$ is equivalent to the following formula internally: 
    \[\forall k\geq_{\stlop{S}} j. \lnot\lnot (\exists \ell\inLopE. (\ell\in_{}\stlop{S} \land k\leq \ell \land \ell\Vdash_{\stlop{S}} \varphi) ).\]
    Furthermore, since $S$ is uniform, 
    Lemma~\ref{lem:uniformstlop} (1) implies that $\intinEff{j\Vdash_{\stlop{S}} \lnot\lnot\varphi}$ is equivalent to the following predicate: 
    \[\bigcap_{k\in S \cap\upset{j}} \lnot\lnot (\bigcup_{\ell\in S \cap\upset{k}} \intinEff{\ell\Vdash_{\stlop{S}} \varphi}).\]
    The assumption on $S$ and $j$ ensures that this last predicate is realizable.
\end{proof}

\begin{proof}[Proof of Theorem~\ref{thm:separation}]
    We define a sequence of partial functions $\{\, f_{n} \,\}_{n\in \N}$ such that each $f_{n}$ has the Turing degree corresponding $\Tjump{n}$, 
    and satisfies $f_{n}\subseteq f_{n+1}$ for all $n$.
    Let $j_{n}\mycoloneqq j_{f_{n}}$ by the construction in Definition~\ref{def:monoandidem}.
    By the first part of Theorem~\ref{thm:specialcase}, the set $S_{n} \mycoloneqq \{\, j_{n}, j_{n+1}\,\}$ forms a uniform subset of $\LopEff$.
    Therefore, regarding the standard lop-frame $\PosetP_{n}\mycoloneqq \stlop{S_{n}}$ for $S_{n}$,
    we observe the following: 
    \begin{enumerate}
        \item Since $S_{n}\subseteq \DLopEff$, Lemma~\ref{lem:stlop} (2) implies that $\Eff\models \PosetP_{n}\subDLopE$. 
        \item Since $j_{n}$ is minimum in $S_{n}$, Lemma~\ref{lem:uniformstlop} (3) implies that $\Eff\models \forall k\in\PosetP_{n}. (j_{n}\leq k)$.
        \item Furthermore, as we have seen in Example~\ref{ex:jrealforSCA}, $\Eff\models \intjtrans{(\AxScheme{\Sigma_{n}}{DNE})}[j_{n}]$ holds.
    \end{enumerate}
    Thus, by Theorem~\ref{thm:sufconforSCA}, for any $j\in S_{n}$ and any $T$ in $\{\, \AxScheme{\Sigma_{n+1}}{DNE},\,  \AxScheme{\Pi_{n+1}\lor \Pi_{n+1}}{DNE} \,\}$, 
    the following equivalence holds: 
    \[T \text{ is } \PosetP_{n}\text{-realizable at } j \iff T \text{ is } j\text{-realizable}.\]
    In particular, by Example~\ref{ex:jrealforSCA}, we obtain that 
    $\AxScheme{\Sigma_{n+1}}{DNE}$ is $\PosetP_{n}$-realizable at $j_{n}$, while $\AxScheme{\Pi_{n+1}\lor \Pi_{n+1}}{DNE}$ is not.
    On the other hand, since $\AxScheme{\Pi_{n+1}\lor \Pi_{n+1}}{DNE}$ is $\PosetP_{n}$-realizable at $j_{n+1}$, 
    it follows from Lemma~\ref{lem:Prealdbnegsnt} that $\lnot\lnot(\AxScheme{\Pi_{n+1}\lor \Pi_{n+1}}{DNE})$ is $\PosetP_{n}$-realizable at $j_{n}$.
    Thus, combining these with the soundness of $\PosetP$-realizability for $\HA$, we obtain the desired separation.
\end{proof}

\begin{remark}\label{rem:PNFT}
    Here, we discuss the proof-theoretic significance of our separation result by explaining why the two axioms, $\AxScheme{\Pi_{n}\lor\Pi_{n}}{DNE}$ and $\AxScheme{\Sigma_{n}}{DNE} + \lnot\lnot(\AxScheme{\Pi_{n}\lor\Pi_{n}}{DNE})$,
    are of interest in intuitionistic proof theory.
    As is well known, the prenex normal form theorem (hereafter the $\mathrm{PNFT}$) does not generally hold in intuitionistic theories.
    Fujiwara and Kurahashi \cite{FujiwaraKurahashi21} investigated the strength of the $\mathrm{PNFT}$ in the hierarchy of semi-classical arithmetic.
    They specifically focused on the classes $\mathrm{E}_{n}$ and $\mathrm{U}_{n}$, which are classically equivalent to $\Sigma_{n}$ and $\Pi_{n}$, respectively.
    In this context, they established the following results:
    \begin{enumerate}
        \item $\HA + \AxScheme{\Pi_{n}\lor\Pi_{n}}{DNE}$ is sufficient to prove the $\mathrm{PNFT}$ for $\mathrm{U}_{n}$ \cite[Theorem~5.3]{FujiwaraKurahashi21}.
        Furthermore, it is necessary in a precise sense \cite[Theorem~7.3]{FujiwaraKurahashi21}.
        \item $\HA + \AxScheme{\Sigma_{n}}{DNE} + \lnot\lnot(\AxScheme{\Pi_{n}\lor\Pi_{n}}{DNE})$ is sufficient to prove the $\mathrm{PNFT}$ for $\mathrm{E}_{n}$ \cite[Theorem~5.3 and Corollary~5.4]{FujiwaraKurahashi21}
        (see also \cite[Remark~5.5]{FujiwaraKurahashi21} for an interesting discussion).
    \end{enumerate}
    Thus, these two theories exhibit distinct properties concerning the $\mathrm{PNFT}$.
    It is known that the underivability of the latter from the former follows from the monotone modified realizability interpretation \cite{ABHK04, Kohlenbachbook}.
    However, to the best of the author's knowledge, the underivability in the opposite direction had not been confirmed.

    As a related problem, an erratum \cite{KohlenbachErratum} to \cite{ABHK04} shows that $\AxScheme{\mathrm{E}_{n}}{DNE}$ does not imply $\AxScheme{\Sigma_{n}}{LLPO}$ over $\HA$ for $n=1$.
    Its proof relies on functional interpretation by bar-recursive functionals.
    To extend this proof to the general case for any $n$, a relativization argument would be required.
    On the other hand, our Theorem~\ref{thm:separation} not only shows this separation as a corollary, but also establishes it for the general case of any $n$
    because $\AxScheme{\mathrm{E}_{n}}{DNE}$ follows from $\AxScheme{\Sigma_{n}}{DNE} + \lnot\lnot(\AxScheme{\Pi_{n}\lor\Pi_{n}}{DNE})$ (by the $\mathrm{PNFT}$), 
    and $\AxScheme{\Sigma_{n}}{LLPO}$ is equivalent to $\AxScheme{\Pi_{n}\lor\Pi_{n}}{DNE}$ under $\AxScheme{\Sigma_{n-1}}{DNE}$.
    As a comparison, the argument in \cite{KohlenbachErratum} is useful in that it establishes the underivability of $\AxScheme{\Sigma_{1}}{LLPO}$ even in the presence of $\mathrm{DNS}$ (double negation shift) for arbitrary formulas.
    Nevertheless, we believe that our proof is more elementary and direct.

\end{remark}

\section{Conclusion}\label{sec:conclusion}

Indeed, Theorem~\ref{thm:separation} was obtained using $\PosetP$-realizability with a standard lop-frame.
While this result is still interesting because its proof relies mostly on syntactic arguments, 
it might not be considered an application beyond the framework of de Jongh-Goodman realizability.
Nevertheless, there are many non-standard lop-frames in the effective topos.
Exploring non-trivial applications based on such non-standard lop-frames is left for future work.

Furthermore, we have restricted our attention to first-order formulas in this paper.
As mentioned in the Introduction, Goodman's theorem, which is one of the important applications of de Jongh-Goodman realizability, 
is a conservation result on Heyting arithmetic in all finite types.
This strongly motivates us to extend our translation to finite-type systems.
For such an extension, we expect that ``change of type'' will be essential (similar to the $\jshfunc$-translation introduced in Section~\ref{subsec:shjtrans}).
The main reason is that the associated sheaf functor does not preserve exponentials in general.
The necessity of such type changes also appears in the previous study on the $j$-translation for G\"{o}del's system $\mathrm{T}$ \cite{Xu20}.
The internal Kripke frame considered in Section~\ref{subsec:internalKripkeframe} and the internal sheaf topos discussed in Section~\ref{subsec:internalsheaftopos} 
will be helpful guides for this extension.
By investigating these semantic structures, we expect to find an appropriate way to extend our translation to finite-type systems, second-order logic, and even higher-order logic.
We leave these topics for future research.


\bibliographystyle{plainurl}
\bibliography{myref_lopframe}

\end{document}